\documentclass[11pt]{amsart}
\usepackage{amsmath}
\usepackage{amssymb}
\usepackage{enumerate}
\usepackage{amsbsy}
\usepackage{amsfonts}
\usepackage{color}

\newtheorem{theorem}{Theorem}[section]
\newtheorem{proposition}[theorem]{Proposition}
\newtheorem{lemma}[theorem]{Lemma}

\newtheorem{cor}[theorem]{Corollary}
\newtheorem{defn}[theorem]{Definition}

\numberwithin{equation}{section}

\theoremstyle{remark}
\newtheorem{remark}{Remark}

\newcommand{\ep}{\varepsilon}

\newcommand{\al}{\alpha}
\newcommand{\pl}{\psi_\lambda}
\newcommand{\ld}{\lambda}
\newcommand{\sai}{S_\al(I)}
\newcommand{\xai}{X_\al(I)}

\begin{document}

\title{Energy concentration of the focusing energy-critical FNLS}

\author[Y. Cho]{Yonggeun Cho}
\address{Department of Mathematics, and Institute of Pure and Applied Mathematics, Chonbuk National University, Jeonju 561-756, Republic of Korea}
\email{changocho@jbnu.ac.kr}

\author[G. Hwang]{Gyeongha Hwang}
\address{Department of Mathematical Sciences, Ulsan National Institute of Science and Technology, Ulsan, 689-798, Republic of Korea}
\email{ghhwang@unist.ac.kr}

\author[Y.-S. Shim]{Yong-Sun Shim}
\address{Department of Mathematics, POSTECH, Pohang 790-784, Republic of Korea}
\email{shim@postech.ac.kr}

\thanks{This research was supported by Overseas Research Grants of LG Yonam Foundation.}

\begin{abstract}
We consider the fractional nonlinear Schr\"odinger equation (FNLS) with general dispersion $|\nabla|^\al$ and focusing energy-critical nonlinearities $-|u|^\frac{2\al}{d-\al}u$ and $-(|x|^{-2\al} * |u|^2) u$. By adopting Kenig-Tsutsumi \cite{mets}, Kenig-Merle \cite{keme} and Killip-Visan \cite{kv} arguments, we show the energy concentration of radial solutions near the maximal existence time. For this purpose we use Sobolev inequalities for radial functions and establish strong energy decoupling of profiles. And we also show that when the kinetic energy is confined the maximal existence time is finite for some large class of initial data satisfying the initial energy $E(\varphi)$ is less than energy of ground state $E(W_\al)$ but $\||\nabla|^\frac\al2 \varphi\|_{L^2} \ge \||\nabla|^\frac\al2 W_\al\|_{L^2}$.
\end{abstract}

\maketitle

\vspace{1cm}

\section{Introduction}

 In this paper we consider the Cauchy problem of the focusing fractional nonlinear Schr\"odinger equations:
\begin{align}\label{main eqn}\left\{\begin{array}{l}
i\partial_tu = |\nabla|^\al u - V(u) u,\;\;
\mbox{in}\;\;\mathbb{R}^{1+d} \;\;d \ge 2,\\
u(x,0) = \varphi(x)\;\; \mbox{in}\;\;\mathbb{R}^d,
\end{array}\right.
\end{align}
where $$V(u) = \left\{\begin{array}{l} |u|^\frac{2\al}{d-\al}\;\; (1 < \al < 2),\\(|x|^{-2\al}*|u|^2)\;\; (1 < \al < \min(2, \frac d2)).\end{array}\right.$$ The equation \eqref{main eqn} is of $\dot{H}^\frac\al2$-scaling invariance (so-called energy-critical). That is, if $u$ is a solution
of \eqref{main eqn}, then for any $\lambda > 0$ the scaled function $u_\lambda$, given by
$$
u_\lambda (t, x) = \lambda^{\frac d2-\frac\al2} u(\lambda^\al \,t, \lambda x ),
$$
is also a solution to \eqref{main eqn}.

The problem \eqref{main eqn} can be easily shown to be well-posedness in $C((-T_*, T^*); H_{rad}^\frac\al2)$ for $\al \in [\frac{2d}{2d-1}, 2)$ and $d > \al$ in the case of power type ($d > 2\al$ in the case of Hartree type) through the radial Strichartz estimate. See Lemma \ref{str} below for Strichartz estimate, and also see Theorem 4.10 of \cite{guwa} and Theorem 5.2 of \cite{chho} for LWP and small data GWP. Here $-T_*, T^* \in [-\infty, +\infty]$ are the maximal existence times and $X_{rad}$ denotes the Banach space $X$ of radially symmetric functions. The solution $u$ satisfies the mass and energy conservation laws: for $t \in (-T_*, T^*)$
\begin{align}\begin{aligned}\label{consv}
&\qquad m(u(t)) = \|u(t)\|^2_{L^2} = m(\varphi), \\  %= \|\varphi\|^2_{L^2}
&E(u(t)) =  \mathcal{K}(u(t)) + \mathcal{V}(u(t)) = E(\varphi),
\end{aligned}\end{align}
where $$\mathcal K(u) = \frac12 \int | |\nabla|^\frac\al2 u(x)|^2\,dx, \;\;\mathcal V(u) = - \frac1{\mu} \int V(u)|u|^2 dx,$$
$\mu = \frac{2d}{d-\al}$ for power type and $\mu = 4$ for Hartree type.

At this point due to the dependency on the profile which is the critical nature of \eqref{main eqn}  we do not know that $\limsup_{t \to T^*}\||\nabla|^\frac\al2 u(t)\|_{L^2}$ is infinity or not. In this paper we want to address some energy concentration phenomena for both cases. We first consider the concentration in the case of unconfined kinetic energy.
\begin{theorem}[Unconfined case]\label{conc-infty}
Let $d > \al+1$ in the case of power type ($d > 2\al$ in the case of Hartree type). Assume that $\varphi \in H^\frac\al2_{rad}$ and $u$ is the unique solution to \eqref{main eqn} in $C([0, T^*); H_{rad}^\frac\al2)$ with
$$
\limsup_{t \to T^*}\||\nabla|^\frac\al2 u(t)\|_{L^2} = +\infty.
$$
Then for any $R > 0$ we have
$$
\limsup_{t\to T^*}\||\nabla|^\frac\al2 u(t)\|_{L^2(|x| \le R)} = +\infty.
$$
Moreover, if for $t < T^*$ $u(t) \in L^\infty$, then
$$
\limsup_{t \to T^*}\|u(t)\|_{L^\infty(|x|\le R)} = +\infty.
$$
The same result holds near $-T_*$.
\end{theorem}

Next we deal with the confined case for which it is necessary to implement subtle estimate associated with ground state. The ground state of \eqref{main eqn} plays an important role. It is a unique positive radial solution of
\begin{align}\label{ell eqn}
|\nabla|^\al W - V(W) W.
\end{align}
%Glassey \cite{gl} introduced first a virial argument to show the
%finite time blow-up of solutions to NLS with pure power type nonlinearity. In \cite{chkl, frohlenz2} the authors extended the finite time blowup to the fractional or high-order NLS with Hartree type nonlinearity. Until now there have not been known the blowup theories of NLS with power type nonlinearity except for the exact blowup solutions or numerical results of fourth-order NLS. See \cite{fip, bfm, bf}.  In this paper we establish the finite time blowup theory of the NLS with general dispersion and focusing inhomogeneous nonlinearity as in \eqref{main eqn} by adopting the virial argument of previous results.
%
%
In \cite{chelo, lieb, dass} the authors showed that any solution of the elliptic equation \eqref{ell eqn} is a constant multiple, dilation and translation of the function $W_\al(x) = C_1(1 + C_2|x|^2)^{-\frac{d-\al}2}$ which is  in $\dot H^\frac\al2$ for $0 < \al < \frac d2$, where $C_1, C_2$ depend on $d, \al$. The solution $W_\al$ is closely related to the best constant problem of the inequality
$$
\int V(u)|u|^2\,dx \le C_{d, \al} \||\nabla|^\frac\al2 u\|_{L^2}^\mu.
$$
Indeed, the maximizer $u \neq 0$ of the above inequality, that is,
\begin{align}\label{bc}
\int V(u)|u|^2\,dx = C_{d, \al} \||\nabla|^\frac\al2 u\|_{L^2}^\mu,
\end{align}
is characterized as $u = e^{i\theta}\lambda^\frac{d-\al}{2}W_\al(\lambda(x-x_0))$ for some $\theta \in [-\pi, \pi]$, $\lambda > 0$ and $x_0 \in \mathbb R^d$. See \cite{au, tal, cota} for power type. We will treat this problem for Hartree type in the appendix. Since $W_\al$ is a solution of \eqref{ell eqn}, $\||\nabla|^\frac\al2 W_\al\|_{L^2}^2 = \int V(W_\al)|W_\al|^2\,dx$. Thus $\||\nabla|^\frac\al2 W_\al\|_{L^2}^2 = C_{d,\al}^{-\frac2{\mu-2}}$.
%The uniqueness follows from the fact that every solution of \eqref{ell eqn} is also the solution of $|\nabla|^\al W - \gamma(n, \al)|W|^\frac{2\al}{n-\al} W$ for some $\gamma$. So is the converse. It is well-known that the solution of the later equation is unique and thus $W_\al$ is unique.

Let us denote by $\sai$ for an interval $I$ the spaces $L_{I, x}^\frac{2(d+\al)}{d-\al}$ for power type and $L_{I}^6L_x^\frac{2d}{d-\frac{4\al}{3}}$ for Hartree type.
Then Lemma \ref{loc} below states that LWP of \eqref{main eqn} in $\dot H^\frac\al2$ holds for some $\al$ and $d$ and also shows the blowup criterion that $\|u\|_{S_\al((-T_*, 0])} = +\infty$ and $\|u\|_{S_\al([0, T^*))} = +\infty$ when $T_*, T^* < +\infty$.
Furthermore by following the arguments in \cite{keme, kv, mgz, gswz} with profile decomposition developed in Section 3 below one can readily get the following theorem.
%================================================================================================================================================
\begin{theorem}\label{gwp}
Let $d \ge 2$, $\frac{2d}{2d-1} < \al < 2$, $\al < d \le 2\al$ for power type ($d > 2\al$  for Hartree type) and let $\varphi \in \dot H^\frac\al2_{rad}$. Assume that
$$
\sup_{t \in (-T_*, T^*)}\||\nabla|^\frac\al2 u(t)\|_{L^2} < \||\nabla|^\frac\al2 W_\al\|_{L^2}.
$$
Then $T_*, T^* = +\infty$ and $\|u\|_{S_\al(\mathbb R)} < +\infty$.
\end{theorem}
%================================================================================================================================================
As a corollary one can show that $T_* = T^* = +\infty$ and $\|u\|_{S_\al(\mathbb R)} < +\infty$ if $E(\varphi) < E(W_\al)$ and $\||\nabla|^\frac\al2 \varphi\|_{L^2} < \||\nabla|^\frac\al2 W_\al\|_{L^2}$. The same result also holds for the defocusing case. The restriction $\al \in (\frac{2d}{2d-1}, 2)$ comes from the optimal range of Strichartz estimates (see Lemma \ref{str}). The condition $\al \le 2\al$ for power type is necessary to estimate perturbation like $\||\sum_j^J f_j|^\frac{2\al}{d-\al}(\sum_j^J f_j) - \sum_j^J |f_j|^\frac{2\al}{d-\al}f_j\|_{\dot H^\frac\al2}$. For this see the arguments below \eqref{app-norm}.

 At this point one may expect the sharpness of Theorem \ref{gwp} and the blowup ($\|u\|_{S_\al((-T_*, T^*))} = +\infty$) when $E(\varphi) < E(W_\al)$ and $\||\nabla|^\frac\al2 \varphi\|_{L^2} \ge \||\nabla|^\frac\al2 W_\al\|_{L^2}$. Unfortunately we do not know the complete answers. We think this is just a technical problem due to non-locality arising when treating $|\nabla|^\al$. However, in case when kinetic energy is confined we can show the energy concentration near the maximal existence time and also find some class of initial data guaranteeing the finite time blowup. We first introduce the energy concentration.
%================================================================================================================================================
\begin{theorem}[Confined case]\label{energy-conc}
Let $d \ge 2$, $\frac{2d}{2d-1} < \al < 2$, $\al < d \le 2\al$ for power type ($d > 2\al$  for Hartree type) and let $\varphi \in \dot H^\frac\al2_{rad}$. Assume that
$$
\|u\|_{S_\al([0, T^*))} = +\infty,\quad \sup_{t \in [0, T^*)}\||\nabla|^\frac\al2 u(t)\|_{L^2} < +\infty.
$$
If $T^*$ is finite, then there exists a sequence $t_n \to T^*$ such that for any sequence $R_n \in (0, \infty)$ obeying $(T^*-t_n)^{-\frac1\al}R_n \to \infty$,
$$
\limsup_{n \to \infty} \int_{|x| \le R_n} ||\nabla|^\frac\al2 u(t_n, x)|^2\,dx \ge \||\nabla|^\frac\al2W_\al\|_{L^2}^2.
$$
The same result also holds near $-T_*$ if $T_* < +\infty$.
\end{theorem}
%================================================================================================================================================
%If $T^* = +\infty$, then with minor modification of proof for $T^* < +\infty$ we can show that there exists a sequence $t_n \to +\infty$ such that for any sequence $R_n \in (0, \infty)$ obeying $t_n^{-\frac1\al}R_n \to \infty$,
%$$
%\limsup_{n \to \infty} \int_{|x| \le R_n} ||\nabla |^\frac\al2 u(t_n, x)|^2\,dx \ge \||\nabla|^\frac\al2 W_\al\|_{L^2}^2.
%$$
The Sch\"{o}dinger case was treated by Killip and Visan in \cite{kv}. Here we adapt their arguments to fractional case with nonlinear profile approximation.
We want to mention that due to the lack of pseudo-conformal symmetry of the equation \eqref{main eqn} we could not get the similar result when the solution blowup at time infinity.

From now on we try to demonstrate some evidence of the finite time blowup. Based on the virial argument the finite time blowup was shown for mass-critical Hartree type fractional Schr\"{o}dinger equations in \cite{chkl} and for fourth order power type NLS \cite{cow}, where the mass-critical nature and radial symmetry are playing a crucial role in the proof of blowup. Those arguments cannot be applied to the power type mass-critical fractional NLS because of the lack of enough cancelation property of nonlinearity for virial argument to proceed. Since we do not know whether the kinetic energy is confined, it is hard to apply them to energy subcritical and mass supercritical or energy critical problem. However, if we are involved in energy critical problem and the energy is confined, then by using Sobolev inequality for radial functions \cite{chooz-ccm} it is plausible to establish the virial argument to get finite time blowup for both power type and Hartree type. The following theorem leads us off the finite time blowup.
%================================================================================================================================================
\begin{theorem}\label{blow}
Let $\varphi \in H_{rad}^\frac\al2$ and $u$ be the unique solution of \eqref{main eqn} in $C([0,T^*); H_{rad}^\frac\al2)$ for the maximal existence time $T^* \in (0, +\infty]$. Suppose that $d \ge 2$, $\frac{4}{3} \le \al < 2$, $\al < d < 3\al$ for power type  and $d > 2\al+2$, $\frac{2d}{2d-1} < \al < 2$ for Hartree type. Then for any $\varphi$ satisfying that
\begin{align}
&\quad\; \||x|\sqrt{1-\Delta} \varphi\|_{L^2} + \||x|^2\varphi\|_{L^2} < + \infty,\label{ass-mo}\\
&E(\varphi) < E(W_\al),\quad \||\nabla|^\frac\al2\varphi\|_{L^2} \ge \||\nabla|^\frac\al2 W_\al\|_{L^2},\label{ass}
\end{align}
if $\sup\limits_{0 \le t < T^*}\||\nabla|^\frac\al2u(t)\|_{L^2} < +\infty$, then $T^* < \infty$.
\end{theorem}
%=================================================================================================================================================
%The first part means that $T^*$ could be $ +\infty$. The second blowup is called type II blowup.

The rest of paper is organized as follows: In Section 2 we gather some preliminary lemmas necessary for the proof of confined energy concentration including the profile decomposition in energy space. In Section 3 we show the energy concentration, Theorems \ref{conc-infty} and \ref{energy-conc}. Section 4 is devoted to proving finite time blowup under energy confinement. In the last section we consider the best constant problem \eqref{bc} for Hartree equation.

\subsection*{Notations}
We will use the notations:\\
$\bullet$ $|\nabla| = \sqrt{-\Delta}$, $\dot H_r^s =
|\nabla|^{-s}L^r$, $\dot H^s=\dot H_2^s$, $H_r^s = (1 -
\Delta)^{-s/2} L^r$, $H^s = H_2^s$, $L^r = L_x^r(\mathbb R^d)$ for some  $s \in \mathbb R$ and $1 \le r \le \infty$.\\
$\bullet$ We use the following mixed-norm notations $L_I^qL^r = L_t^q(I; L_x^r(\mathbb R^d))$, $L_{I, x}^q = L_I^qL^q$ and $L_t^qL^r = L_{\mathbb R}^qL^r$.\\
$\bullet$ $\widehat f(\xi) = \int_{\mathbb R^d} e^{-ix\cdot \xi}f(x)\,dx$.\\
$\bullet$ For any dyadic number $N$ we denote frequency localization of function $f$ by $f_N$, which is defined by $\widehat{f_N}(\xi) = \widehat{P_Nf}(\xi) = \beta(\xi/N)\widehat{f}$ for a fixed Littlewood-Paley function $\beta \in C_{0, rad}^\infty$ with $\beta \widetilde \beta = \beta$ and $P_N\widetilde P_N = P_N$, where $\widetilde \beta(\xi) = \beta(\xi/2) + \beta(\xi) + \beta(2\xi)$ and $\widetilde P_N = P_{N/2} + P_N + P_{2N}$.\\
%$A \lesssim B$ and $A \gtrsim B$ means that $A \le CB$ and
%$A \ge C^{-1}B$, respectively for some $C>0$.
$\bullet$ As usual different positive
constants are denoted by the same letter $C$, if not specified. \\
$\bullet$ $[A, B]$ denotes the commutator $AB - BA$ for any operators $A$ and $B$ defined on suitable Banach spaces.\\
$\bullet$ $\big<u, v\big> = \int_{\mathbb R^d} u \,\overline v\, dx$ and $\big<f\,;\, g \big> = \sum_{1 \le j \le d}\big<f_j, g_j \big>$ for $f = (f_1,\cdots, f_d), g = (g_1, \cdots, g_d)$.

\section{Preliminary lemmas}
We define the linear propagator $U(t)$ of the linear equation $iu_t = |\nabla|^\al u$ with initial datum $f$. Then it is
formally given  by
\begin{align}\label{int eqn}
U (t)f = e^{-it|\nabla|^\al} f = \frac1{(2\pi)^d} \int_{\mathbb{R}^d} e^{i(x\cdot \xi - t|\xi|^\alpha)}\widehat{f}(\xi)\,d\xi.
\end{align}
We have Strichartz estimates for radial functions (see \cite{cholee} and \cite{guwa, ke}) as follows.
\begin{lemma}\label{str}
Suppose that $d \ge 2$, $\frac{2d}{2d-1} \le \al < 2$ and $f, F$ are radial. Then there hold
$$
\|U(t)f\|_{L_I^qL_x^r} \le C\|f\|_{L^2}, \quad \|\int_0^t U(t-t')F(t')\,dt'\|_{L_I^qL_x^r} \le C\|F\|_{L_I^{\widetilde q'}L_x^{\widetilde r'}}
$$
for the pairs $(q, r)$ and $(\widetilde q, \widetilde r)$ such that
$$
\frac\al{q} + \frac{d}r = \frac d2,\;\; 2 \le q, r \le \infty,\;\; (q, r) \neq (2 , \frac{4d-2}{2d-3}).
$$
\end{lemma}
\noindent Such pairs are said to be $\al$-admissible.

Then we have the following inverse Strichartz estimate.
\begin{lemma}\label{inverse str}
Fix $d \ge 2$ for power type and $d > 2\al$ for Hartree type. Let $f \in \dot H_{rad}^\frac\al2$ and $\eta > 0$ such that
$$
\|U(t)f\|_{\sai} \ge \eta
$$
for some interval $I \subset \mathbb R$. Then there exists $\widetilde C = \widetilde C(\||\nabla|^\frac\al2 f\|_{L^2}, \eta)$, and $J \subset I$ so that
$$
\int_{|x| \le \widetilde C|J|^\frac1\al} \Big|U(t)|\nabla|^\frac\al2 f \Big|^2\,dx \ge \widetilde{C}^{-1}\;\;\mbox{for all}\;\;t \in J.
$$
Here $\widetilde C$ does not depend on $I$ or $J$.
\end{lemma}
\begin{proof}[Proof of Lemma \ref{inverse str}]
For simplicity we only consider the Hartree type, the power type can be treated similarly to \cite{kv}. We will show that
\begin{align}\label{choice m}
\|U(t)f_M\|_{\sai} \ge C^{-1}\eta^\frac{3d}{4\al}\||\nabla|^{\frac\al2}f\|_{L^2}^{1-\frac{3d}{4\al}}
\end{align}
for some dyadic $M \ge A|I|^{-\frac1\al}$ and some (depending on $d$ and $\al$). Here $A = C\||\nabla|^\frac\al2 f\|_{L^2}^{-D_1}\eta^{D_2}$ and the constants $D_1, D_2$ will be specified later.

We assume that \eqref{choice m} is true. By Strichartz estimate we have
$$
\|U(t)f_M\|_{L_I^\frac{6(d-\al)}{d} L^\frac{2(d-\al)}{d-\frac{4\al}{3}}} \le C M^{-\frac\al2}\||\nabla|^\frac\al2 f\|_{L^2}.
$$
Combining this with \eqref{choice m}, we get by H\"{o}lder's inequality that
$$
\|U(t)f_M\|_{L_{I,x}^\infty}  \ge  C^{-1}\||\nabla|^\frac\al2 f\|_{L^2}^{1-\frac{3d^2}{4\al^2}} \eta^{\frac{3d^2}{4\al^2}} M^{\frac{d-\al}{2}}.
$$
From this with the fact that the kernel of $M^\frac\al2 |\nabla|^{-\frac\al2} \widetilde P_M$ is integrable and its value is independent of $M$ we deduce that
$$
C^{-1}\||\nabla|^\frac\al2 f\|_{L^2}^{1-\frac{3d^2}{4\al^2}}\eta^{\frac{3d^2}{4\al^2}} M^{\frac{d}{2}} \le \|(M^\frac\al2|\nabla|^{-\frac\al2}\widetilde P_M)(U(t)|\nabla|^\frac\al2 f_M)\|_{L_{I, x}^\infty} \le C\|U(t)|\nabla|^\frac\al2 f_M\|_{L_{I, x}^\infty}.
$$
Thus there exist $(t_0, x_0) \in I \times \mathbb R^d$ so that
\begin{align}\label{lower}
|(U(t_0)|\nabla|^\frac\al2f_M(x_0)| \ge A_0M^\frac d2,
\end{align}
where $A_0 = C^{-1}\||\nabla|^\frac\al2 f\|_{L^2}^{1-\frac{3d^2}{4\al^2}}\eta^{\frac{3d^2}{4\al^2}}$. Let $A_1 = \frac{A_0}{2C\||\nabla|^\frac\al2 f\|_{L^2}}$. Then
for $|x-x_0| \le A_1M^{-1} $ and $|t-t_0| \le A_1M^{-\al}$ we have
\begin{align*}
|U(t_0)|\nabla|^\frac\al2 f_M(x_0) - U(t) |\nabla|^\frac\al2  f_M(x)| \le \frac12A_0 M^\frac d2
\end{align*}
and thus
\begin{align*}
|U(t) |\nabla|^\frac\al2  f_M(x)| \ge \frac12A_0 M^\frac d2.
\end{align*}
This yields for all $t \in J = \{t \in I : |t-t_0| \le A_1M^{-\al}\}$
$$
\int_{|x-x_0| \le A_1M^{-1}} |U(t) |\nabla|^\frac\al2 f_M(x)|^2 \,dx \ge \frac{s_d}4 A_0^2A_1^d,
$$
where $s_d$ is the measure of the unit ball.
By convexity we have $$|U(t)|\nabla|^\frac\al2 f_M|^2 \le C_0|U(t)|\nabla|^\frac\al2 f|^2 * \beta^*_M,$$ where $\beta^*_M(x) = M^d |\widehat \beta(Mx)|$ and $C_0 = \int \beta^* \,dx$. And also
\begin{align*}
\int_{|x-x_0| \le A_1M^{-1}} |U(t) |\nabla|^\frac\al2 f_M(x)|^2 \,dx \le C_0\big\langle |U(t)|\nabla|^\frac\al2 f|^2, \beta_M^* * \chi_{\{|x-x_0| \le A_1M^{-1}\}} \big\rangle.
\end{align*}
%We assume that $|\widehat\beta(x)| \le C_\beta(1 + |x|)^{-1}$
%If $|x - x_0| > A_2 A_1M^{-1}$ for $A_2 > 1 + \max(A_1^{-1}, \frac{CC_\beta A_1^{d-2}}{8A_0^2})$, where $C$ is the constant appearing in \eqref{lower}, then
We divide inner product into two parts as follows:
$$
\big\langle |U(t)|\nabla|^\frac\al2 f|^2, \beta_M^* * \chi_{\{|x-x_0| \le A_1M^{-1}\}} \big\rangle \le I + II,
$$
where
\begin{align*}
&I = \big\langle |U(t)|\nabla|^\frac\al2 f|^2, \chi_{\{|x-x_0| \le A_2A_1M^{-1}\}}\beta_M^* * \chi_{\{|x-x_0| \le A_1M^{-1}\}} \big\rangle,\\
&II = \big\langle |U(t)|\nabla|^\frac\al2 f|^2, \chi_{\{|x-x_0| > A_2A_1M^{-1}\}}\beta_M^* * \chi_{\{|x-x_0| \le A_1M^{-1}\}} \big\rangle.
\end{align*}
Now we can find a constant $A_2 = A_2(\||\nabla|^\frac\al2f\|_{L^2}, \eta) > 1$\footnote{We may choose $A_2$ as $A_2 > 1 + \max(A_1^{-1}, \frac{C_0C_\beta A_1^{d-2}}{8A_0^2})$, where $C_\beta$ is the constant satisfying $\beta^*(x) \le C_\beta(1 + |x|)^{-1}$.} such that $C_0II \le \frac{s_d}{8}A_0^2A_1^d$. Then $I \le C_0\int_{|x-x_0| \le A_2 A_1M^{-1}} |U(t) |\nabla|^\frac\al2 f(x)|^2 \,dx$ and thus
we obtain
\begin{align}\label{lower1}
\int_{|x-x_0| \le A_2 A_1M^{-1}} |U(t) |\nabla|^\frac\al2 f(x)|^2 \,dx \ge \frac{s_d}{8C_0^2}A_0^2A_1^d.
\end{align}

On the other hand, since $f$ is radial, we use the Sobolev inequality \cite{chooz-ccm} that
\begin{align}\label{sobo-radial}
\sup_{x\in \mathbb R^d}|x|^{\frac{d-\al}2}|f(x)| \le C\||\nabla|^\frac\al2 f\|_{L^2}\;\;\mbox{a.e.}
\end{align}
 together with \eqref{lower} to get
$$
A_0(M|x_0|)^\frac{d-\al}2 \le |x_0|^\frac{d-\al}2M^{-\frac\al2}|U(t_0)|\nabla|^{\frac\al2}f_M(x_0)| \le C\||\nabla|^{\frac\al2} f\|_{L^2},
$$
which means that
\begin{align}\label{radius}
|x_0| \le A_0^{-\frac{2}{d-\al}}M^{-1}.
\end{align}
Combining \eqref{radius} with \eqref{lower1}, we have
$$
\int_{|x| \le (A_0^{-\frac2{d-\al}} + A_2A_1)M^{-1}} |U(t) |\nabla|^\frac\al2 f(x)|^2 \,dx \ge \frac{s_d}{8C_0^2}A_0^2A_1^d.
$$
Since $|I| \ge A^\al M^{-\al}$, we deduce that $|J| \ge \frac1{2}\min(A^{\al}, A_1)M^{-\al}$. By letting $\widetilde{C} = \max(2(A_0^{-\frac2{d-\al}} + A_1A_2)\max(A^{-1}, A_1^{-\frac1\al}), \frac{8C_0^2}{s_d A_0^2A_1^d})$ we get the desired result.

Now we show \eqref{choice m}. %We will implement this by dividing into three cases: (i) $d \ge 5\al$, (ii) $\frac{10\al}3 \le d < 5\al$, (iii) $2\al < d < \frac{10\al}3$ .
%\subsection*{(i) $d \ge 5\al$}
By Littlewood-Paley theory and H\"older's inequality we have
\begin{align*}
\eta^6 &\le \|U(t)f\|_{\sai}^6\\
& \le C\int_I \left(\int(\sum_N |U(t)f_N|^2)^\frac{d}{d-4\al/3}\,dx\right)^\frac{3d-4\al}{d}\,dt \\
&\le C\sum_{M\le N}\int_I\left( \int|U(t)f_M|^\frac{d}{d-4\al/3} |U(t)f_N|^\frac{d}{d-4\al/3}\,dx\right)^\frac{3d-4\al}{d}\,dt\\
&\le C\sum_{M\le N}\int_I \|U(t)f_M\|_{L^\frac{2d}{d-2\al}}^\frac{3d-4\al}{d}\|U(t)f_M\|_{L^\frac{2d}{d-4\al/3}}^\frac{4\al}{d}
\|U(t)f_N\|_{L^\frac{2d}{d-4\al/3}}^\frac{4\al}{d}\|U(t)f_N\|_{L^\frac{2d}{d-2\al/3}}^\frac{3d-4\al}{d}\,dt\\
&\le C\sum_{M\le N} (\frac{M}{N})^\frac{\al(3d-4\al)}{3d}(\||\nabla|^\frac\al2 f_M\|_{L^2}\||\nabla|^\frac\al2f_N\|_{L^2})^\frac{3d-4\al}{d}(\|U(t)f_M\|_{\sai} \|U(t)f_N\|_{\sai})^\frac{4\al}{d}\\
&\le C\sup_{M}\|U(t)f_M\|_{\sai}^\frac{8\al}{d}\left(\sum_{M\le N} (\frac{M}{N})^\frac\al3 \||\nabla|^\frac\al2f_M\|_{L^2}\||\nabla|^\frac\al2f_N\|_{L^2}\right)^\frac{3d-4\al}{d}\\
&\le C\sup_{M}\|U(t)f_M\|_{\sai}^\frac{8\al}{d}\||\nabla|^\frac\al2f\|_{L^2}^\frac{6d-8\al}{d}.
\end{align*}

From the Sobolev inequality it follows that
$$
\|U(t)f_M\|_{\sai} \le C(|I|M^\al)^\frac16\||\nabla|^\frac\al2f\|_{L^2}.
$$
Thus we conclude that there exists $M \ge A |I|^{-\frac1\al}$ so that \eqref{choice m} holds with $A = C\||\nabla|^\frac\al2 f\|_{L^2}^{-\frac{9d}{2\al^2}}\eta^\frac{9d}{2\al^2}$.%, 3, \frac{3d}{2(d-5\al/3)}$ for each case.

\end{proof}

Next we introduce the tightness of trajectories of solution. The proof is almost same as the one of Proposition 2.13 in \cite{kv} and so we omit it.
\begin{lemma}\label{tightness}
Let $v : I \times \mathbb R^d \to \mathbb C$ be a radial solution to \eqref{main eqn} with $\|v\|_{\sai} < \infty$. Suppose that
$$
\int_{|x| \le r_k} |U(t_k)|\nabla|^\frac\al2 v(\tau_k)|^2\,dx \ge \ep
$$
for some $\ep > 0$, $r_k > 0$, and bounded sequences $t_k \in \mathbb R$ and $\tau_k \in I$. Then
$$
\Big| \||\nabla|^\frac\al2 v(\tau_k)\|_{L^2}^2 - \int_{|x| \le a_k r_k}|U(t_k)|\nabla|^\frac\al2v(\tau_k)|^2\,dx \Big| \to 0
$$
for any sequence $a_k \to +\infty$.
\end{lemma}

We close this section by introducing local well-posedness and stability. Since the proof is quite standard, we omit the details (for instance see \cite{keme, chho}).
\begin{lemma}\label{loc}
Let $\al \in (\frac{2d}{2d-1}, 2)$ and $\al < d < 3\al$ for power type ($d > 2\al$ for Hartree type), and let $\varphi \in \dot H_{rad}^\frac\al2$, $\|\varphi\|_{\dot H^\frac\al2} \le A$. Then there exists $\delta = \delta(A)$ such that if $\|U(t-t_0)\varphi\|_{\sai} \le \delta$, $t_0 \in I$, there exists a unique solution $u \in C(I; \dot H_{rad}^\frac\al2)$ to \eqref{main eqn} with
$$
\sup_I\|u(t)\|_{\dot H^\frac\al2} + \||\nabla|^\frac\al2 u\|_{\xai} \le C(A),\quad  \|u\|_{\sai} \le 2\delta.
$$
Here $\xai = L_I^\frac{2(d+\al)}{d-\al}L^\frac{2d(d+\al)}{d^2+\al^2}$ for power type and $L_I^6L^\frac{2d}{d-\al/3}$ for Hartree type. Moreover, $\varphi \mapsto u \in C(I; \dot H^\frac\al2)$ is Lipschitz. If $A$ is sufficiently small, then $I = \mathbb R$.
\end{lemma}

\begin{lemma}\label{stab}
Assume that $\al \in (\frac{2d}{2d-1}, 2)$ and $\al < d < 3\al$ for power type ($d > 2\al$ for Hartree type).
Let $I = [0, L), L \le +\infty$, and let $\widetilde u$ be radial and defined on $\mathbb R^d \times I$ be such that
$$
\sup_{t \in I}\|\widetilde u(t)\|_{\dot H^\frac\al2} \le A, \;\;\|\widetilde u\|_{\sai} \le M,\;\; \||\nabla|^\frac\al2 \widetilde u\|_{\xai} < \infty
$$
for some constants $A$ and $M$, and $\widetilde u$ verifies in the sense of integral equation
$$
i\widetilde u_t = |\nabla|^\al \widetilde u - V(\widetilde u)\widetilde u + e
$$
for some function $e$. Let $\varphi \in \dot H_{rad}^\frac\al2$ be such that $\|\varphi - \widetilde u(0)\|_{\dot H^\frac\al2} \le A'$. Suppose there exists $\varepsilon_0 = \varepsilon_0(M, A, A')$ such that if $0 < \varepsilon \le \varepsilon_0$ and
$$
\||\nabla|^\frac\al2 e\|_{Y_\al(I)} \le \varepsilon, \;\;\|U(t)(\varphi - \widetilde u(0))\|_{\sai} \le \varepsilon,
$$
then there exists a unique radial solution $u$ on $\mathbb R^d \times I$ to \eqref{main eqn} such that
$$
\|u\|_{\sai} + \sup_I\|u(t) - \widetilde u(t)\|_{\dot H^\frac\al2} \le C(M, A, A').
$$
Here $Y_\al(I) = L_{I}^2L_x^\frac{2d}{d+\al}$ for both power type and for Hartree type.
\end{lemma}

%
%\section{Profile decomposition}
Now we consider the profile decomposition in energy space. Most of them are standard and thus we only show the energy decoupling of Hartree case.

\begin{lemma}[see Theorem 1 of \cite{cox}]\label{dispest}
Let $\{t_n\}$ be sequence in $\mathbb{R}$. Suppose $\lim_{n \rightarrow \infty} |t_n| = \infty$, then for any $f \in C^\infty_0$
\[\|U(t_n)f\|_{L^p} \rightarrow 0 \text{\ as\ }n \rightarrow \infty,\]
when $p >2$.
\end{lemma}
%
%\begin{proof}
%From triangle inequality, we have
%\[\|U(t_n)f\|_{L^p_x} \leq  \|U(t_n)P_{\leq N}f\|_{L^p_x} + \|U(t_n)P_{> N}f\|_{L^p_x}\]
%for all $N>0$.
%Then for given $\ep$, $\|U(t_n)P_{> N}f\|_{L^p_x} \leq \ep$ for sufficiently large $N$. And by using dispersive estimates, we get $\|U(t_n)P_{\leq N}f\|_{L^p_x} \leq \ep$ for sufficiently %large $n$.
%\end{proof}

 The profile decomposition of $U(t)$ for mass critical case was already verified for radial data in \cite{chkl2} (see also \cite{chkl3}). From that decomposition, one can easily prove the following profile decomposition for the energy critical case:
\begin{lemma}\label{main}
Let  $d \ge 2$, $\frac{2d}{2d-1} < \alpha < 2$, and $(q, r)$ be $\al$-admissible pairs with $2 < q, r <
\infty$. Suppose that
$\{u_n\}_{n \geq 1}$ is a sequence of complex-valued radial functions bounded in $\dot{H}^\frac \al2$. Then up to a subsequence, for
any $J \geq 1$, there exist a sequence of radial functions
$\{\phi^j\}_{1\leq j \leq J}\subset \dot{H}^{\frac \al2}$, $\omega_n^J \in \dot{H}^{\frac \al2}$ and a
family of parameters $(h_n^j, t_n^j)_{1 \leq j \leq J, n \geq 1}$
such that
\begin{align}\label{profileid}
u_n(x) = \sum_{1 \leq j \leq J} U(t^j_n)[(h^j_n)^{-d/2+\al/2}\phi^j(\cdot/{h^j_n})](x) + \omega^J_n(x)
\end{align}
and the following properties are satisfied:
\begin{align}\label{profilerem}
\lim_{J \rightarrow \infty} \limsup_{n \rightarrow \infty}
\||\nabla|^\frac \al2U(\cdot)
\omega^J_n\|_{L^q_tL^r_x} = 0,
\end{align} and for $j \neq k$, $(h^j_n, t^j_n)_{n \geq
1}$ and $(h^k_n, t^k_n)_{n \geq 1}$ are asymptotically orthogonal in
the sense that
\begin{align}\begin{aligned}\label{profileparaortho}
&\mbox{either}\;\; \limsup_{n \rightarrow
\infty}\left(\frac{h^j_n}{h^k_n} +
\frac{h^k_n}{h^j_n}\right) = \infty,\\
&\mbox{or}\;\; (h^j_n) = (h^k_n)\;\;\mbox{and}\;\;\limsup_{n
\rightarrow \infty}\frac{|t^j_n - t^k_n|}{(h^j_n)^\alpha} = \infty,
\end{aligned}\end{align}
and for each $J$
\begin{align}\label{profilenormortho}
\lim_{n \rightarrow \infty} \Big[\|u_n\|_{\dot{H}^\frac \al2}^2 - (\sum_{1 \leq j \leq J}\|\phi^j\|_{\dot{H}^\frac \al2}^2
+ \|\omega^J_n\|_{\dot{H}^\frac \al2}^2) \Big] = 0.
\end{align}
\end{lemma}

\begin{remark}
Since the space and frequency translations do not appear in the profile decomposition due to the radial symmetry, it is possible to get the strong convergence of remainder term in $L^q_t\dot{H}_r^{\frac \al2}$ as in \eqref{profilerem} not in $L^q_tL^{\frac {rd}{d-\frac {r\al}{2}}}_x$ norm. It plays a crucial role in the proof of Theorem \ref{energy-conc}.
\end{remark}

From energy critical profile decomposition, we prove some useful corollaries.
\begin{cor}\label{strong dec}
Suppose that $\{u_n\}_{n \geq 1}$ is a sequence of complex-valued radial functions bounded in $\dot{H}^\frac \al2$. Let $\{\phi^j\}_{1\leq j \leq J}\subset \dot{H}^{\frac \al 2}$, $\omega_n^J \in \dot{H}^{\frac \al 2}$ and a family of parameters $(h_n^j, t_n^j)_{1 \leq j \leq J, n \geq 1}$ from Lemma \ref{main}. Define group operator $G_n^j$ as $G_n^j(f) = U(t^j_n)[(h^j_n)^{-d/2+\al/2}f(\cdot/{h^j_n})](x)$. Then we have
\begin{align}\begin{aligned}\label{weak-limit-profile}
&(G_n^j)^{-1}(\omega_n^J) \rightharpoonup 0 \text{\ weakly\ in\ }\dot{H}^{\frac \al2}\text{\ as\ }n\ \rightarrow \infty,\\
&(G_n^j)^{-1}(u_n) \rightharpoonup \phi^j \text{\ weakly\ in\ }\dot{H}^{\frac \al2}\text{\ as\ }n\ \rightarrow \infty.
\end{aligned}\end{align}
\end{cor}

\begin{proof}
We first prove
\[(G_n^j)^{-1}U(t)(|\nabla|^\frac \al2u_n) \rightharpoonup U(t)|\nabla|^\frac \al2\phi^j \text{\ weakly\ in\ }L^{\frac {2(2+\al)}{d}}_{t,x}\text{\ as\ }n\ \rightarrow \infty.\]
Applying $(G_n^j)^{-1}U(t)$ to \eqref{profileid}, we obtain
\[(G_n^j)^{-1}(U(t)|\nabla|^\frac \al2u_n) = U(t)|\nabla|^{\frac \al 2}\phi^j + \sum_{j'\neq j}^J (G_n^j)^{-1}(G_n^{j'})(U(t)|\nabla|^\frac \al2\phi^i) + (G_n^j)^{-1}U(t)|\nabla|^\frac \al2\omega^J_n.\]
From the pairwise orthogonality of the family of parameters, we have
\[(G_n^j)^{-1}(G_n^{j'})(U(t)|\nabla|^\frac \al2\phi^j) \rightharpoonup 0 \text{\ weakly\ in\ }L^{\frac {2(2+\al)}{d}}_{t,x}\text{\ as\ }n\ \rightarrow \infty\]
for every $j' \neq j$. Let $\omega^J$ be the weak limit of $\{(G_n^j)^{-1}U(t)|\nabla|^\frac \al2\omega^J_n\}$. Then
\[(G_n^j)^{-1}(U(t)|\nabla|^\frac \al2u_n) \rightharpoonup U(t)|\nabla|^\frac \al2\phi^j + \omega^J.\]
Since the weak limit is unique, $\omega^J$ does not depend on $J$.
And from
\[\|(G_n^j)^{-1}U(t)|\nabla|^\frac \al2\omega^J_n\|_{L^{\frac {2(2+\al)}{d}}_{t,x}} \leq \limsup_{n \rightarrow \infty}
\|U(t)|\nabla|^\frac \al2\omega^J_n\|_{L^{\frac {2(2+\al)}{d}}_{t,x}} \longrightarrow 0 \text{ as }J\rightarrow \infty,\]
we have $\omega^J = 0$ for every $J \geq 1$.
So we have
\[(G_n^j)^{-1}U(t)(|\nabla|^\frac \al2u_n) \rightharpoonup U(t)|\nabla|^\frac \al2\phi_j \text{\ weakly\ in\ }L^{\frac {2(2+\al)}{d}}_{t,x}\text{\ as\ }n\ \rightarrow \infty.\]
Then following lemma gives the conclusion.
\begin{lemma}[Lemma 3.63 in \cite{meve}]
Let $\{v_n\}$ and $v$ be in $L^2$. The following statements are equivalent.
\begin{enumerate}
\item $v_n \rightharpoonup v$ weakly in $L^2$.
\item $U(t)v_n \rightharpoonup U(t)v$ weakly in $L^{\frac {2(2+\al)}{d}}_{t,x}$.
\end{enumerate}
\end{lemma}
\end{proof}

\newcommand{\sar}{S_\al(\mathbb R)}
\newcommand{\xar}{X_\al(\mathbb R)}

\begin{proposition}
Let $\{u_n\}_{n \geq 1}$ be a sequence of complex-valued radial functions
satisfying
\[\|u_n\|_{\dot{H}^\frac \al2} \leq A \text{\quad and \quad} \|U(t)u_n\|_{\sar} \geq \delta.\]
Suppose $\{\phi^j\}_{1\leq j \leq J}\subset \dot{H}^{\frac \al 2}$ be linear profiles obtained in Lemma \ref{main}.
Then there exist at least one linear profile $\phi^{j_0}$ such that
\[\|U(t)\phi^{j_0}\|_{\sar} \geq C(A,\delta).\]
\end{proposition}

\begin{proof}
By Lemma \ref{main}, we have
\[U(t)(u_n)(x) = \sum_{1 \leq j \leq J} U(t - t^j_n)[(h^j_n)^{-d/2+\al/2}\phi^j(\cdot/{h^j_n})](x) + U(t)\omega^J_n(x)\]
with
\[\lim_{J \rightarrow \infty} \limsup_{n \rightarrow \infty}
\||\nabla|^\frac \al2U(t)
\omega^l_n\|_{\xar} = 0,\]
and for each $J$
\[\lim_{n \rightarrow \infty} \Big[\|u_n\|_{\dot{H}^\frac \al2}^2 - (\sum_{1 \leq j \leq J}\|\phi^j\|_{\dot{H}^\frac \al2}^2
+ \|\omega^J_n\|_{\dot{H}^\frac \al2}^2) \Big] = 0.\]
From the orthogonality(see Lemma 3.3 in \cite{chkl2}), we get
\[\lim_{n \rightarrow \infty} \|\sum_{j=1}^J U(t - t^j_n)[(h^j_n)^{-d/2+\al/2}\phi^j(\cdot/{h^j_n})](x)\|_{\sar}^4 = \sum_{j=1}^J \|U(t)(\phi^j)(x)\|_{\sar}^4\]
for every $J \geq 1$.
However,
\begin{align*}
&\limsup_{n \rightarrow \infty} \|U(t)(u_n)(x) - \sum^J_{j=1}U(t - t^j_n)[(h^j_n)^{-d/2+\al/2}\phi^j(\cdot/{h^j_n})](x)\|_{\sar} \\
&\qquad \leq \limsup_{n \rightarrow \infty}\|U(t)
\omega^J_n\|_{\sar} \leq \limsup_{n \rightarrow \infty}\||\nabla|^\frac \al2U(t)
\omega^J_n\|_{\xar} \to 0\;\;\mbox{as}\;\;J \to \infty.
\end{align*}
So we obtain
\[\limsup_{n \rightarrow \infty} \|U(t)(u_n)(x)\|_{\sar}^4 = \sum_{j=1}^J \|U(t)(\phi^j)(x)\|_{\sar}^4.\]
And Strichartz estimate gives
\[\sum_{j=1}^J \|U(t)(\phi^j)(x)\|_{\sar}^4 \le C \big(\sup_{j \geq J}\|U(t)(\phi^j)(x)\|_{\sar}^2\big)\sum_{j \geq 1}\|\phi^j\|_{\dot{H}^\frac \al2}^2.\]
Since $\sum_{j \geq 1}\|\phi^j\|_{\dot{H}^\frac \al2}^2 \leq \limsup_{n \rightarrow \infty} \|u_n\|_{\dot{H}^\frac \al2}^2 \leq A^2$, we have
\[\sup_{j \geq 1}\|U(t)(\phi^j)(x)\|_{\sar}^2 \geq \frac {\delta^4}{A^2}.\]
In particular, we can find $j_0$ such that
\[\|U(t)(\phi^{j_0})(x)\|_{\sar}^2 \geq \frac {\delta^4}{A^2}.\]
\end{proof}

\begin{proposition}
Let $\{u_n\}_{n \geq 1}$ be a sequence of complex-valued radial functions bounded in $\dot H^\frac\al2$. Suppose $\{\phi^j\}_{1\leq j \leq J}\subset \dot{H}^\frac\al2$ be linear profiles obtained in Lemma \ref{main}. Then for each $J$,
\[\lim_{n \rightarrow \infty} \left(E(u_n) - \sum_{1 \leq j \leq J}E(U(t^j_n)[(h^j_n)^{-d/2+\al/2}\phi^j(\cdot/{h^j_n})](x)) - E(\omega^J_n)\right) = 0.\]
\end{proposition}

\begin{proof} For the power type we refer the readers to the proof in \cite{ker}. We only consider the Hartree case. Also see \cite{mgz} for NLS with Hartree nonlinearity.

Thanks to the kinetic energy decoupling \eqref{profilenormortho}, it suffices to show
\begin{align*}
\lim_{n \rightarrow \infty} \Big(\int |u_n|^2(|x|^{-2\al} * &|u_n|^2)dx - \sum_{1 \leq j \leq l}\int |G_n^j(\phi^j)|^2(|x|^{-2\al} * |G_n^j(\phi^j)|^2)dx\\
&- \int |\omega^J_n|^2(|x|^{-2\al} * |\omega^J_n|^2)dx\Big) = 0.
\end{align*}
We first prove
\begin{align*}
\lim_{n \rightarrow \infty} \Big(\int |u_n|^2(|x|^{-2\al} * |u_n|^2)dx &- \int |u_n - G_n^1(\phi^1)|^2(|x|^{-2\al} * |u_n - G_n^1(\phi^1)|^2)dx\\
-&\int |G_n^1(\phi^1)|^2(|x|^{-2\al} * |G_n^1(\phi^1)|^2)dx \Big)= 0.
\end{align*}
Then repeated arguments give the conclusion.

When $\lim_{n \rightarrow \infty} \big|\frac {t^1_n}{(h^1_n)^\al}\big| = \infty$, we obtain
\begin{align*}
&\quad \lim_{n \rightarrow \infty} \int |G_n^1(\phi^1)|^2(|x|^{-2\al} * |G_n^1(\phi^1)|^2)dx = \lim_{n \rightarrow \infty} \|G_n^1(\phi^1)\|^4_{L^{\frac {2d}{d-\al}}} = \lim_{n \rightarrow \infty} \Big\|U\Big(\frac {t^1_n}{(h^1_n)^\al}\Big)(\phi^1)\Big\|^4_{L^{\frac {2d}{d-\al}}} = 0
\end{align*}
by using H\"older inequality, fractional integration, scaling and Lemma \ref{dispest}. Similarly, one can prove \[\lim_{n\to \infty}\left(\int |u_n|^2(|x|^{-2\al} * |u_n|^2)dx - \int |u_n - G_n^1(\phi^1)|^2(|x|^{-2\al} * |u_n - G_n^1(\phi^1)|^2)dx\right) = 0.\]

Now we handle the case $\lim_{n \rightarrow \infty} \big|\frac {t^1_n}{(h^1_n)^\al}\big| < \infty$. By taking subsequence we may assume that $\lim_{n \rightarrow \infty} \frac {t^1_n}{(h^1_n)^\al} = t_\infty.$ Let $S^1_n(u_n) := (h^1_n)^{\frac d2 - \frac \al 2}u_n(h^1_n \cdot)$. Then we have
\begin{align*}
S^1_n(u_n) \rightharpoonup U(t_\infty)\phi^1 &\text{\ weakly\ in\ }\dot{H}^{\frac \al2} \text{\ as\ }n \rightarrow \infty\\
\text{\ and\ }G_n^1(S^1_n(\phi_1)) \rightarrow U(t_\infty)\phi^1 &\text{\ strongly\ in\ }L^{\frac {2d}{d-\al}} \text{ as }n \rightarrow \infty.
\end{align*}
The scaling symmetry yields
\begin{align*}
&\int |u_n|^2(|x|^{-2\al} * |u_n|^2)dx - \int |u_n - G_n^1(\phi^1)|^2(|x|^{-2\al} * |u_n - G_n^1(\phi^1)|^2)dx\\
&\qquad -\int |G_n^1(\phi^1)|^2(|x|^{-2\al} * |G_n^1(\phi^1)|^2)dx\\
&=\int |S^1_n(u_n)|^2(|x|^{-2\al} * |S^1_n(u_n)|^2)dx\\
&\qquad-\int |S^1_n(u_n) - G_n^1(S^1_n(\phi^1))|^2(|x|^{-2\al} * |S^1_n(u_n) - G_n^1(S^1_n(\phi^1))|^2)dx\\
&\qquad -\int |G_n^1(S^1_n(\phi^1))|^2(|x|^{-2\al} * |G_n^1(S^1_n(\phi^1))|^2)dx\\
&=: I_n + II_n + III_n,
\end{align*}
where
\begin{align*}
I_n &= \int |S^1_n(u_n)|^2(|x|^{-2\al} * |S^1_n(u_n)|^2)dx\\
&\quad- \int |S^1_n(u_n) - U(t_\infty)\phi^1|^2(|x|^{-2\al} * |S^1_n(u_n) - U(t_\infty)\phi^1|^2)dx\\
&\quad- \int |U(t_\infty)\phi^1|^2(|x|^{-2\al} * |U(t_\infty)\phi^1|^2)dx,\\
II_n &= \int |S^1_n(u_n) - G_n^1(S^1_n(\phi_1))|^2(|x|^{-2\al} * |S^1_n(u_n) - G_n^1(S^1_n(\phi_1))|^2)dx\\
&\quad - \int |S^1_n(u_n) - U(t_\infty)\phi^1|^2(|x|^{-2\al} * |S^1_n(u_n) - U(t_\infty)\phi^1|^2)dx,\\
III_n &= \int |G_n^1(S^1_n(\phi_1))|^2(|x|^{-2\al} * |G_n^1(S^1_n(\phi_1))|^2)dx\\
 &\quad- \int |U(t_\infty)\phi^1|^2(|x|^{-2\al} * |U(t_\infty)\phi^1|^2)dx.
\end{align*}
$I_n$ goes to $0$ by Lemma \ref{refatou} below. And by using H\"older inequality, fractional integration and Lemma \ref{dispest} again, we also obtain
\begin{align*}
\lim_{n\to \infty}(II_n + III_n) = 0.
\end{align*}
\end{proof}

\begin{lemma}\label{refatou}
Let $\{f_n\}$ be bounded sequence in $\dot{H}^{\frac \al2}$. If $f_n$ weakly converges to $f$, then for some subsequence $\{f_n\}$,
\begin{align*}
\int |f_n|^2(|x|^{-2\al} * |f_n|^2) - &|f_n - f|^2(|x|^{-2\al} * |f_n - f|^2) - |f|^2(|x|^{-2\al} * |f|^2)dx\\
& \longrightarrow 0 \text{ as } n \rightarrow \infty.
\end{align*}
\end{lemma}

\begin{proof}
Assume that $\||\nabla|^\frac\al2 f_n\|_{L^2} \le M$ for all $n \ge 1$. Since $C^\infty_{0}$ is dense in $\dot{H}^{\frac \al2}$, one can find $\beta \in C^\infty_{0}$ such that $\|\beta - f\|_{\dot{H}^{\frac \al2}} < \frac{\ep}{12C(1+M)^2}$ for some constant $C$. And since the multiplication operator $T_{\beta} : \dot{H}^{\frac \al2} \rightarrow L^p_x, T_{\beta}(f) = \beta f$ is compact when $\beta \in C^\infty_0, 1 \leq p < \frac {2d}{d-\al}$, there exists subsequence of $f_n$ such that $\|\beta(f_n-f)\|_{L^{\frac {d}{d-\al}}} < \frac{\ep}{12C(1+M)^2}$ if $n \ge N$.

On the other hand, one can easily check that
\begin{align*}
\qquad&\int |f_n|^2(|x|^{-2\al} * |f_n|^2) - |f_n - f|^2(|x|^{-2\al} * |f_n - f|^2) - |f|^2(|x|^{-2\al} * |f|^2)dx\\
%= &\int |f_n|^2(|x|^{-2\al} * |f_n|^2) - |f_n - f|^2(|x|^{-2\al} * |f_n|^2) - |f|^2(|x|^{-2\al} * |f_n|^2)\\
%& -|f_n - f|^2(|x|^{-2\al} * |f_n - f|^2) + |f_n - f|^2(|x|^{-2\al} * |f_n|^2) - |f_n - f|^2(|x|^{-2\al} * |f|^2)\\
%& -|f|^2(|x|^{-2\al} * |f|^2) + |f_n - f|^2(|x|^{-2\al} * |f|^2) + |f_n|^2(|x|^{-2\al} * |f|^2)\\
%& -|f_n|^2(|x|^{-2\al} * |f|^2)+ |f|^2(|x|^{-2\al} * |f_n|^2)dx\\
=&\int (f\overline{f_n} + \overline{f}f_n - 2|f|^2)(|x|^{-2\al} * |f_n|^2) - |f_n - f|^2(|x|^{-2\al} *(f\overline{f_n} + \overline{f}f_n - 2|f|^2))\\
& +(f\overline{f_n} + \overline{f}f_n - 2|f|^2)(|x|^{-2\al} *|f|^2 )-|f_n|^2(|x|^{-2\al} * |f|^2)+ |f|^2(|x|^{-2\al} * |f_n|^2)dx.
\end{align*}
Let us observe that
\[\int -|f_n|^2(|x|^{-2\al} * |f|^2)+ |f|^2(|x|^{-2\al} * |f_n|^2)\,dx = 0.\]
Then by using H\"older's inequality and Sobolev embedding, we have
\begin{align*}
&\qquad \int(f\overline{(f_n-f)})(|x|^{-2\al} * |f_n|^2)dx\\
& = \int((f - \beta)\overline{(f_n-f)})(|x|^{-2\al} * |f_n|^2) + (\beta\overline{(f_n-f)})(|x|^{-2\al} * |f_n|^2)dx\\
&\le C \|f - \beta\|_{\dot H^\frac\al2}\|f_n-f\|_{L^{\frac {2d}{d-\al}}}\|f_n\|_{L^{\frac {2d}{d-\al}}} + C \|\beta(f_n-f)\|_{L^{\frac {d}{d-\al}}}\|f_n\|_{L^{\frac {2d}{d-\al}}}\\
&\le CM^2\|f - \beta\|_{\dot H^\frac\al2} + CM\|\beta(f_n-f)\|_{L^{\frac {d}{d-\al}}} < \frac\ep6\;\;\mbox{if}\;\;n \ge N.
\end{align*}
We need to treat remaining 5 terms. But they can be done by exactly the same way as above.
\end{proof}

Using the local well-posedness theorem with initial data at $t=0$ or $t=\pm \infty $, we define the nonlinear profile by the maximal nonlinear solution for each linear profile.

\begin{defn}
Let $\{(h_n,t_n)\} $ be a family of parameters and $\{t_n\}$ have a limit in $[-\infty, \infty]$. Given a linear profile $\phi \in \dot{H}^\frac \al2$ with $\{(h_n,t_n)\}$, we define the nonlinear profile associated with it to be the maximal solution $v$ to \eqref{main eqn} which is in $C((-T_{\min},T_{\max}); \dot{H}^\frac \al2)$ satisfying an asymptotic condition: For the sequence $\{t_n\}$,
\[\lim_{n \to \infty}\|U(t_n)\phi - v(t_n)\|_{\dot{H}^\frac \al2} = 0.\]
\end{defn}
\begin{remark}\label{nonlinear prof}
Let $\{u_n\}_{n \geq 1}$ be a sequence of complex-valued radial functions bounded in $\dot H^\frac\al2$ and $\{\phi^j\}_{1\leq j \leq J}\subset \dot{H}^\frac\al2$ be the corresponding linear profiles obtained in Lemma \ref{main}. Then by refining subsequence and using diagonal argument we may assume that for each $j$ the sequence $\{t_n^j\}$ converges to $t^j \in [-\infty, +\infty]$. By using the standard time-translation and absorbing error we may assume that $t^j := 0$ and either $t_n^j := 0 $ or $t_n^j \to \pm \infty$.

As stated in \cite{keme} the nonlinear profiles $v^j : I^j \times \mathbb R^d \to \mathbb C$ associated with $\phi^j$ and $t_n^j$ always exist and they can be summarized as follows:
If $t_n^j = 0$, then $v^j$ is the maximal solution to \eqref{main eqn}  with initial data $v^j(0) = \phi^j$. If $t_n^j \to \pm\infty$, then $v^j$ is the maximal solution to \eqref{main eqn} that scatters forward/backward in time to $U(t)\phi^j$.

\end{remark}

\section{Energy concentration}
In this section we show Theorems \ref{conc-infty} and \ref{energy-conc} by following the arguments as in \cite{mets} and \cite{kv}, respectively.

\subsection{Unconfined kinetic energy: Proof of Theorem \ref{conc-infty}}
Let $\beta$ be a $C_0^\infty$-bump function which is $1$ for $|x| \le 1$ and $0$ for $|x| > 1$. Then we have from Lemma \ref{com-lem} and mass conservation that
\begin{align*}
\|\beta(\cdot/R)u\|_{L^\frac{2d}{d-\al}}^\frac{2d}{d-\al} &\le C\||\nabla|^\frac\al2(\beta(\cdot/R)u)\|_{L^2}^\frac{2d}{d-\al}\\
 &\le C\|\big[ |\nabla|^\frac\al2, \beta(\cdot/R)\big] u\|_{L^2}^\frac{2d}{d-\al} + C\||\nabla|^\frac\al2 u\|_{L^2(|x| \le 2R)}^\frac{2d}{d-\al}\\
 & \le CR^{-\frac\al2}\|u\|_{L^2}^\frac{2d}{d-\al} + C\||\nabla|^\frac\al2 u\|_{L^2(|x| \le 2R)}^\frac{2d}{d-\al}\\
 &  \le C(R^{-\frac\al2}\|\varphi\|_{L^2}^\frac{2d}{d-\al}) + C\||\nabla|^\frac\al2 u\|_{L^2(|x| \le 2R)}^\frac{2d}{d-\al}\\
 & \le A + \||\nabla|^\frac\al2 u\|_{L^2(|x| \le 2R)}^\frac{2d}{d-\al}
\end{align*}
for some $A = A(R, \|\varphi\|_{L^2})$. Using the endpoint Sobolev inequality (Proposition 2 of \cite{chooz-ccm}) and real interpolation \cite{bl} that
$$
|x|^\frac{d-1}2 |f(x)| \le C \|f\|_{\dot B_{2,1}^\frac12} \le C\|f\|_{L^2}^\frac{\al-1}\al\||\nabla|^\frac\al2 f\|_{L^2}^\frac1\al,
$$
we have
\begin{align*}
\|(1-\beta(\cdot/R))u\|_{L^\frac{2d}{d-\al}}^\frac{2d}{d-\al} &= \int (1-\beta(\cdot/R))|u|^{\frac{2\al}{d-\al}}|u|^2dx &\le CR^{-\frac{\al(d-1)}{d-\al}}\|u\|_{L^2}^\frac{2(d-1)}{d-\al}\||\nabla|^\frac\al2u\|_{L^2}^{\frac{2}{d-\al}} \le A \||\nabla|^\frac\al2u\|_{L^2}^{\frac{2}{d-\al}},
\end{align*}
where $\|f\|_{\dot B_{2,1}^\frac12} := \sum_N N^\frac12 \|f_N\|_{L^2}$ is the homogeneous Besov norm.

On the other hand, for Hartree type we have  that
\begin{align*}
\int V(u)&|\beta(x/R)u|^2\,dx\\
 &\le C\|u\|_{L^\frac{2d}{d-\al}}^2\|\beta(\cdot/R)u\|_{L^\frac{2d}{d-\al}}^2\\
 &\le C\|\beta(\cdot/R)u\|_{L^\frac{2d}{d-\al}}^4 + C\|(1 - \beta(\cdot/R))u\|_{L^\frac{2d}{d-\al}}^2\|\beta(\cdot/R)u\|_{L^\frac{2d}{d-\al}}^2\\
 &\le A + \||\nabla|^\frac\al2 u\|_{L^2(|x| \le 2R)}^4 + A\||\nabla|^\frac\al2 u\|_{L^2}^\frac{2}{d}(A + \||\nabla|^\frac\al2 u\|_{L^2(|x| \le 2R)}^2)
\end{align*}
and
\begin{align*}
\int V(u)&|(1-\beta(x/R))u|^2\,dx\\
 &\le C\|u\|_{L^\frac{2d}{d-\al}}^2\|(1-\beta(\cdot/R))u\|_{L^\frac{2d}{d-\al}}^2\\
 &\le C\|(1-\beta(\cdot/R))u\|_{L^\frac{2d}{d-\al}}^4 + C\|(1 - \beta(\cdot/R))u\|_{L^\frac{2d}{d-\al}}^2\|\beta(\cdot/R)u\|_{L^\frac{2d}{d-\al}}^2\\
 &\le A\||\nabla|^\frac\al2 u\|_{L^2}^\frac{4}{d} + A\||\nabla|^\frac\al2 u\|_{L^2}^\frac{2}{d}(A + \||\nabla|^\frac\al2 u\|_{L^2(|x| \le 2R)}^2).
\end{align*}

From the energy conservation it follows that
\begin{align*}
\||\nabla|^\frac\al2u(t)\|_{L^2}^2 &= \frac{2}{\mu}\left\{\begin{array}{l}\|\beta(\cdot/R)u(t)\|_{L^\frac{2d}{d-\al}}^\frac{2d}{d-\al} + \|(1-\beta(\cdot/R))u(t)\|_{L^\frac{2d}{d-\al}}^\frac{2d}{d-\al}\\
\int V(u)|\beta(\cdot/R)u(t)|^2\,dx + \int V(u)|(1-\beta(\cdot/R))u(t)|^2\,dx\end{array}\right\} + E(\varphi).
\end{align*}
Let $y(t) = \||\nabla|^\frac\al2u(t)\|_{L^2}^2$ and $z(t) = \||\nabla|^\frac\al2 u(t)\|_{L^2(|x|\le 2R)}^2$. Then from the above estimates we have
$$
y(t) \le  C\left\{\begin{array}{l} A + z^\frac{d}{d-\al} + y^\frac{1}{d-\al}\\ A + Ay^\frac1d(A+z)+Ay^\frac2d+z^2\end{array}\right\}  + E(\varphi).
$$
Since $\limsup_{t\to T^*}y(t) = +\infty$ and $d > \al+1$ for power type ($d > 2\al$ for Hartree type), we conclude that
$\limsup_{t\to T^*} z(t) = +\infty$.

If $u(t) \in L^\infty$ for all $t < T^*$, then since $\|\beta(\cdot/R)u(t)\|_{L^\frac{2d}{d-\al}} \le CR^\frac{d-\al}2 \|u\|_{L^\infty(|x| \le 2R)}$, by replacing $\||\nabla|^\frac\al2u(t)\|_{L^2(|x|\le 2R)}$ with $\|u\|_{L^\infty(|x| \le 2R)}$ in the above estimates we get the desired.

\subsection{Confined kinetic energy: Proof of Theorem \ref{energy-conc}}
 Choose a sequence $t_n \to T^*$ and let $u_n$ be the solution on $[0, T^* - t_n)$ to \eqref{main eqn} with initial data $u(t_n)$. Then since $\sup_{0 < t < T^*}\||\nabla|^\frac\al2u(t)\|_{L^2} =: M < +\infty$, by Lemma \ref{main} we can decompose each $u_n(0)$ by
$$
u_n(0) = \sum_{j = 1}^J G_n^j \phi^j + \omega_n^J.
$$
We denote the symmetry operator $g_n^j$ by $g_n^jf(t, x) = (h_n^j)^{-\frac{d-\al}{2}}f(t/(h_n^j)^\al, x/h_n^j)$. Then $G_n^j \phi^j = g_n^j U(t_n^j)\phi^j$. Let $v^j : I^j \times \mathbb R^d \to \mathbb C$ be nonlinear profile associated with $\phi^j$ and $(h_n^j, t_n^j)$ as stated in Remark \ref{nonlinear prof}. For each $j, n \ge 1$, we define $v_n^j : I_n^j \times \mathbb R^d \to \mathbb C$ by
$$
v_n^j(t) := g_n^jv^j(\cdot + t_n^j)(t),
$$
where $I_n^j = \{t \in \mathbb R : (h_n^j)^{-\al}t + t_n^j \in I^j\}$. Then $v_n^j$ is also a solution to \eqref{main eqn} with initial data $v_n^j(0) = g_n^jv^j(t_n^j)$ and maximal time interval $I_n^j = (-T_{n,j}^-, T_{n,j}^+)$ for $0 < T_{n,j}^-, T_{n,j}^+ < +\infty$. By the kinetic energy decoupling \eqref{profilenormortho} there exists $J_0 = J_0(\delta_0) \ge 1$ such that
$$
\||\nabla|^\frac\al2 \phi^j\|_{L^2} \le \delta_0\;\;\mbox{for all}\;\;j \ge J_0.
$$
For sufficiently small $\delta_0$, Lemma \ref{loc} yields that $v_n^j$ are global and satisfy that
\begin{align}\label{energy bound}
\sup_{t\in\mathbb R}\||\nabla|^\frac\al2 v_n^j(t)\|_{L^2} + \|v_n^j\|_{\sar} \le C\||\nabla|^\frac\al2 \phi^j\|_{L^2}.
\end{align}

Now we can find a so-called bad profile $\phi^{j_0}$, $1 \le j_0 < J_0$ such that
\begin{align}\label{bad}
\limsup_{n \to \infty}\|v_n^{j_0}\|_{S_\al([0, T^*-t_n)} = +\infty,
\end{align}
\newcommand{\san}{S_\al([0, T_n^*))}
\newcommand{\xan}{X_\al([0, T_n^*))}
\begin{proof}[Proof of \eqref{bad}]
We will actually show that
\begin{align}
\limsup_{n \to \infty}\|v_n^{j_0}\|_{S_\al([0, T_n^*))} = +\infty,
\end{align}
where $T_n^* = \min_{1 \le j < J_0}(T^*-t_n, T_{n,j}^+)$. Suppose that $\limsup_{n \to \infty}\|v_n^j\|_{S_\al([0, T_n^*))} < +\infty$ for all $1 \le j < J_0$. Then this implies that $T^*-t_n \le T_{n, j}^+$ for all $1 \le j < J_0$ if $n$ is large. If $T_{n, j}^+ \le T^*-t_n$ for some $j$, then since $\limsup_{n \to \infty}\|v_n^j\|_{S_\al([0, T_{n,j}^+))} < +\infty$, the maximality means that $T_{n,j}^+ = +\infty$ for sufficiently large $n$. This contradicts the fact $T^* < +\infty$.
 Then from this together with \eqref{energy bound} and \eqref{profilenormortho} it follows that
\begin{align}\label{s-bound}
\sum_{j \ge 1}^J \|v_n^j\|_{S_\al([0, T_n^*))}^2 \le C (1 + \sum_{j \ge J_0}^J\||\nabla|^\frac\al2 \phi^j\|_{L^2}^2) \le C(1 + M^2)
\end{align}
for any $J$ and for sufficiently large $n$. We now define functions $u_n^J$ on $[0, T_n^{m_0}]$ approximating $u_n$ by
$$
u_n^J = \sum_{j = 1}^Jv_n^j + U(t)\omega_n^J.
$$
Since $v^j$ are nonlinear profile associated with $(\phi^j, t_n^j)$, we have
\begin{align*}
\|u_n^J(0) - u_n(0)\|_{\dot H^\frac\al2} = \|\sum_{j = 1}^J(g_n^j v^j(t_n^j) - g_n^j U(t_n^j)\phi^j)\|_{\dot H^\frac\al2} \le \sum_{j=1}^J\|v^j(t_n^j) - U(t_n^j)\phi^j\|_{\dot H^\frac\al2} \to 0
\end{align*}
as $n \to \infty$. By \eqref{profilerem} and \eqref{s-bound} we also have
\begin{align}\begin{aligned}\label{app-s-bound}
\lim_J\limsup_{n \to \infty} \|u_n^J\|_{S_\al([0, T_n^*))} &\le \lim_J\limsup_{n \to \infty}(\|\sum_j v_n^j\|_{\san} + \|U(t)\omega_n^J\|_{\san})\\
& \le C(1 + M^2).
\end{aligned}\end{align}
By the local well-posedness we deduce that
\begin{align}\label{app-norm}
\lim_J\limsup_{n \to \infty} (\|u_n^J\|_{L^\infty_{[0, T_n^*)} \dot H^{\frac \al2}} + \||\nabla|^\frac\al2 u_n^J\|_{\xan}) \le C(M).
\end{align}

On the other hand, $u_n^J$ satisfy that
$$
i\partial_t u_n^J = |\nabla|^\al u_n^J - V(u_n^J)u_n^J + e,
$$
where $e = e_1 + e_2$, $$e_1 = V(u_n^J) - V(\sum_{j=1}^J v_n^j)(\sum_{j=1}^J v_n^j)$$ and $$e_2 = V(\sum_{j=1}^Jv_n^j)(\sum_{j=1}^Jv_n^j)- \sum_{j = 1}^J V(v_n^j)v_n^j.$$
We first show that $\limsup_{n \to \infty}\||\nabla|^\frac\al2 e_2\|_{Y_\al([0, T_n^*))} = 0$.
In fact, from direct calculation we get that for power type
\begin{align*}
e_2 = \frac{2\al}{d-\al}{\rm Re}\sum_{j\neq j'}v_n^{j'}v_n^j\int_0^1| s \sum_{j'\neq j}v_n^{j'} + v_n^j|^\frac{4\al-2d}{d-\al}(s \sum_{j'\neq j}v_n^{j'} + v_n^j)\,ds
\end{align*}
and for Hartree type
\begin{align*}
e_2 = \sum_{j'\neq j}(|x|^{-2\al} * |v_n^{j'}|^2)v_n^j + \sum_j\sum_{j_1' \neq j_2'} (|x|^{-2\al} * (v_n^{j_1'}\overline{v_n^{j_2'}})) v_n^j.
\end{align*}
Since $\al < d \le 2\al$ for power type, we have
\begin{align*}
\||\nabla|^\frac\al2 e_2\|_{Y_{\al}([0,T^*))} \le C\sum_{j\neq j'}\Big(&\||\nabla|^\frac\al2(v_n^{j'}v_n^j)\|_{L^\frac{d+\al}{d-\al}_{[0,T_n^*)}L^\frac{2d(d+\al)}{2d^2-\al d+\al^2}}(\sum_{j=1}^J\|v_n^j\|_{\san}^\frac{3\al-d}{d-\al})\\
& + (\sum_{j=1}^J\||\nabla|^\frac\al2 v_n^j\|_{\xan})(\sum_{j=1}^J\|v_n^j\|_{\san}^\frac{4\al-2d}{d-\al})\|v_n^{j'}v_n^j\|_{L_{[0,T_n^*),x}^\frac{d+\al}{d-\al}}\Big).
\end{align*}
Thus the orthogonality \eqref{profileparaortho} gives
$$
\limsup_{n\to\infty}\||\nabla|^\frac\al2 e_2\|_{L^2_{[0, T_n^*)} L^\frac{2d}{d+\al}} = 0.
$$
For Hartree type by the orthogonality \eqref{profileparaortho} and the argument used for the proof of Lemma 3.3 in \cite{chkl2} one can easily get
\begin{align*}
\limsup_{n\to \infty}\||\nabla|^\frac\al2 e_2\|_{L^2_{[0,T_n^*)}L^\frac{2d}{d+\al}} = 0.
\end{align*}
Now let us consider $e_1$. Let $V_n^J = \sum_{j= 1}^Jv_n^j$ and let us invoke that $\mu = \frac{2d}{d-\al}$ for power type and $\mu = 4$ for Hartree type. Then we have
\begin{align*}
&\||\nabla|^\frac\al2 e_1\|_{Y_\al([0, T_n^*))}\\
 &\le C\Big(\||\nabla^\frac\al2 u_n^J\|_{\xan} + \||\nabla|^\frac\al2V_n^J\|_{\xan}\Big)\Big(\|u_n^J\|_{\san}^{\mu-3} + \|V_n^J\|_{\san}^{\mu-3}\Big)\|U(t)\omega_n^J\|_{\san}\\
&\quad + C\Big(\|u_n^J\|_{\san}^{\mu-2} + \|V_n^J\|_{\san}^{\mu-2}\Big)\||\nabla|^\frac\al2 U(t)\omega_n^J\|_{\xan}.
\end{align*}
By \eqref{profilerem} we get
$$
\lim_{J\to \infty}\limsup_{n\to \infty}\||\nabla|^\frac\al2 e_1\|_{Y_\al([0, T_n^*))} = 0.
$$

We apply Lemma \ref{stab} with $\widetilde u = u_n^J$ and $u = u_n$ to conclude that
$$
\|u_n\|_{S_\al([0,T^*-t_n))} < +\infty \;\;\mbox{for sufficiently large}\;\;n.
$$
This contradicts that $u$ blows up within finite time $T^*$.
\end{proof}

By reordering we may assume that $\limsup_{n\to\infty}\|v_n^1\|_{S_\al([0,T^*-t_n))} = +\infty$ and that there exists $1 \le J_1 < J_0$ such that
$$
\limsup_{n\to\infty}\|v_n^j\|_{S_\al([0, T^*-t_n))} = \infty\;\;(j \le J_1)\;\;\mbox{and}\;\;\limsup_{n\to\infty}\|v_n^j\|_{S_\al([0, T^*-t_n))} < \infty\;\;(j > J_1).
$$
Then for each $m, n \ge 1$, there exist $1 \le j(m, n) \le J_1$ and $0 < T_n^m < T^*-t_n$ such that
\begin{align}\label{pigeonhole}
\sup_{1 \le j \le J_1}\|v_n^j\|_{S_\al([0, T_n^m])} = \|v_n^{j(m,n)}\|_{S_\al([0,T_n^m])} = m.
\end{align}
By using the pigeonhole principle and then reordering, we may assume that $j(m,n) = 1$ for infinitely many $m, n$. Then by Theorem \ref{gwp} there exists $0 \le \tau_n^m \le T_n^m$ such that
$$
\limsup_{m\to\infty}\limsup_{n\to\infty}\||\nabla|^\frac\al2 v_n^1(\tau_n^m)\|_{L^2} \ge \||\nabla|^\frac\al2 W_\al\|_{L^2}.
$$
For any $\ep > 0$ we can find $m_0 = m_0(\ep)$ such that
$$
\||\nabla|^\frac\al2 v_n^1(\tau_n^{m_0})\|_{L^2} \ge \||\nabla|^\frac\al2 W_\al\|_{L^2} - \ep\;\;\mbox{for infinitely many}\;\;n.
$$
\newcommand{\tno}{\tau_n^{m_0}}
\newcommand{\tnm}{\tau_n^-}
\newcommand{\tnp}{\tau_n^+}
Passing to a subsequence we may have that
\begin{align}\label{energy-ep}
\||\nabla|^\frac\al2 v_n^1(\tno)\|_{L^2} \ge \||\nabla|^\frac\al2 W_\al\|_{L^2} - \ep\;\;\mbox{for all}\;\;n \;\;\mbox{and}\;\;\lim_{n\to \infty}\||\nabla|^\frac\al2 v_n^1(\tno)\|_{L^2}\;\;\mbox{exists}.
\end{align}

Now we choose a small $\eta$ to be specified later and fix $n$. Then since $\|v_n^1\|_{S_\al([0, T_n^{m_0}])} = m_0$, we can find $\tnm, \tnp$ with $0 \le \tnm \le \tno \le \tnp\le  T_n^{m_0}$ such that
\begin{align}\label{s-eta}
\|v_n^1\|_{S_\al([\tnm, \tnp])} = \eta.
\end{align}
Using local well-posedness (Lemma \ref{loc}) we get
$$
\|U(t)v_n^1(\tno)\|_{S_\al([\tnm-\tno, \tnp-\tno])} \ge C\eta^{\widetilde{D}}
$$
for some dimension-dependent constant $\widetilde D$. By Lemma \ref{inverse str} there exists $\tnm - \tno \le s_n \le \tnp-\tno$ such that
\begin{align}\label{tightness0}
\int_{|x| \le \widetilde C|T^*-t_n'|^\frac1\al} |U(s_n)|\nabla|^\frac\al2 v_n^1(\tno)|^2\,dx \ge \widetilde C^{-1},
\end{align}
where $\widetilde C = \widetilde C(d, M, \eta)$ and $t_n' = t_n + s_n + \tno$.

From the definition of $v_n^1$ and \eqref{tightness0} we deduce that
$$
\int_{|y| \le \widetilde C(h_n^1)^{-1}|T^*-t_n'|^\frac1\al} |U(s_n(h_n^1)^{-\al})\big(|\nabla|^\frac\al2 v^1((h_n^1)^{-\al} \tno + t_n^1, y)\big)|^2\,dy \ge \widetilde C^{-1}.
$$
By applying Lemma \ref{tightness} and rescaling we have
\begin{align}\label{tightness1}
\Big|\||\nabla|^\frac\al2 v_n^1(\tno)\|_{L^2}^2 - \int_{|x| \le R_n} |U(s_n)v_n^1(\tno)|^2\,dx\Big| \to 0
\end{align}
for any sequence $R_n \in (0, \infty)$ such that $(T^*-t_n')^{-\frac1\al}R_n \to \infty$ as $n \to \infty$.  Let $u_n^J$ be the approximate functions defined on $[0, T_n^{m_0}]$ as above. Then in view of the proof of \eqref{bad} and \eqref{pigeonhole} we can deduce that
$$
\lim_{J \to \infty}\limsup_{n \to \infty}\||\nabla|^\frac\al2(u_n^J(s_n + \tno) - u(t_n'))\|_{L^2} = 0.
$$
%=======================================================================================================================
\newcommand{\bl}{\big\langle}
\newcommand{\br}{\big\rangle}
Using \eqref{profileparaortho} and Corollary \ref{strong dec} we have
$$
\limsup_{n\to\infty}\bl |\nabla|^\frac\al2 u_n^J(s_n + \tno), |\nabla|^\frac\al2v_n^1(s_n +\tno) \br = \limsup_{n\to\infty}\||\nabla|^\frac\al2 v_n^1(s_n+\tno)\|_{L^2}^2
$$
for all $J \ge 1$. Thus we obtain
$$
\limsup_{n\to\infty} |\bl |\nabla|^\frac\al2 u_n(t_n'), |\nabla|^\frac\al2 v_n^1(s_n + \tno) \br| = \limsup_{n\to\infty} \||\nabla|^\frac\al2 v_n^1(s_n+\tno)\|_{L^2}^2.
$$
From \eqref{s-eta} and Strichartz estimate it follows that
$$
\||\nabla|^\frac\al2(v_n^1(s_n + \tno) - U(s_n)v_n^1(\tno))\|_{L^2} \le C\eta^{\mu-2}.
$$
So, if $\eta$ is sufficiently small, then we get
$$
\limsup_{n\to\infty}|\bl|\nabla|^\frac\al2 u_n(t_n'), U(s_n)|\nabla|^\frac\al2 v_n^1(\tno) \br| \ge \lim_{n\to\infty}\||\nabla|^\frac\al2 v_n^1(\tno)\|_{L^2}^2 - \eta^{D'},
$$
for some $D' < \mu-2$.
Therefore by Cauchy-Schwarz inequality and \eqref{energy-ep} we obtain that
\begin{align*}
\limsup_{n\to \infty}\int_{|x| \le R_n} ||\nabla|^\frac\al2 u(t_n')|^2\,dx &\ge \frac{(\lim_{n\to\infty}\||\nabla|^\frac\al2 v_n^1(\tno)\|_{L^2}^2 - \eta^{D'})^2}{\lim_{n\to\infty} \||\nabla|^\frac\al2 v_n^1(\tno)\|_{L^2}^2} \\
&\ge \||\nabla|^\frac\al2 W_\al\|_{L^2}^2 - \ep - 2\eta^{D'} + \eta^{2D'}/M^2.
\end{align*}
Since $\ep$ and $\eta$ can be taken arbitrarily small, we get the desired result.

\section{Proof of finite time Blowup}

Let us denote $\sup_{0 \le t < T^*}\||\nabla|^\frac\al2 u(t)\|_{L^2}$ by $M$ and $\|\varphi\|_{L^2}$ by $m$. We will show that $T^* = T^*(\varphi, M) < +\infty$. From the regularity persistence it follow that if $\varphi \in H^2$, then $u \in C([0, T^*); H^2)$ (this is the case for the power type since $\al < d < 3\al$ and thus $\frac{2\al}{d-\al} > 1$). Since the maximal existence time $T^* = T^*(\varphi)$ is lower semi-continuous, that is, if $\varphi_k \to \varphi$ in $H^\frac\al2$, then $T^*(\varphi) \le \liminf_{k\to \infty} T^*(\varphi_k)$, we may assume that $u \in C([0,T^*); H^{2})$ and $\varphi$ satisfies the condition \eqref{ass}.

%In view of the proof of estimates of moment for the mass-critical case as in \cite{chkl}, one can implement the exactly the same moment estimate to the energy-critical case to show that the solution has high regularity and fast decay if so does the initial data. So by the well-posedness we may assume that the initial data is Schwartz function satisfying \eqref{ass}.

%We first show the well-posedness of \eqref{main eqn}. Due to the singularity of nonlinear term it is not easy to show the contraction on a Strichartz space. Instead, we use regularizing method to find a weak solution and Strichartz estimate to show the uniqueness. Once the uniqueness is obtained, the local well-posedness follows from the continuity argument (see \cite{caz}).

\newcommand{\ald}{\mathcal A}

\newcommand{\blb}{\big[}
\newcommand{\brb}{\big]}
\newcommand{\re}{\mbox{Re}}
\newcommand{\mlf}{\mathbf{m}_{1,\ld}}
\newcommand{\mlft}{\widetilde{\mathbf{m}}_{1, \ld}}
\newcommand{\mls}{\mathbf{m}_{2, \ld}}
\newcommand{\im}{\mbox{Im}}

\subsection{Moment estimates}

\begin{proposition}\label{moment}
If $\varphi$ satisfies the condition \eqref{ass-mo}, then the solution $u \in C([0,T^*); H^2)$ satisfies that for each $t \in (0, T^*)$
\begin{align*}
\||x|u(t)\|_{L^2} \le CMt + \||x|\varphi\|_{L^2},\;\;\||x||\nabla|^{\al-1}u(t)\|_{L^2} + \||x|^2u(t)\|_{L^2} < +\infty.
\end{align*}
\end{proposition}
\begin{proof}
For a fixed radial bump function $\psi \in C_0^\infty$ with $\psi(x) = 1$ when $|x| \le 1$ and $\psi(x) = 0$ when $|x| \ge 2$ we denote $\psi(\frac{x}{\lambda})$ by $\pl$ for $\lambda \ge 1$. Then we can define moments $\mlf, \mlft, \mls$ by
\begin{align*}
&\mlf^2 : = \bl x\pl u ; x\pl u \br,\\
&\mlft^2 : = \bl x \pl|\nabla|u; x\pl |\nabla|u \br,\\
&\mls^2 := \bl |x|^2\pl u, |x|^2\pl u\br.
\end{align*}
Differentiating $\mlf^2$ w.r.t $t$, we have
\begin{align*}
\frac{d}{dt}\mlf^2 &= 2\im\bl |\nabla|^{-\frac\al2}x \pl (|\nabla|^\al u - V(u)u) ; |\nabla|^\frac\al2x \pl u\br = 2\im \bl |\nabla|^{-\frac\al2}x\pl |\nabla|^\al u ; |\nabla|^\frac\al2 x\pl u\br\\
&= 2 \sum_{j=1}^d\im \bl \blb x_j\pl, |\nabla|^\al \brb u, x_j\pl u\br \le 2\left(\sum_j\|\blb x_j\pl, |\nabla|^\al \brb u\|_{L^2}^2\right)^\frac12\mlf.
\end{align*}
In order to estimate the last term we use the following lemma.
\begin{lemma}\label{com-lem}
Let $\beta_\ld(x) = \beta(\frac{x}{\ld})$ for $\beta \in C_0^\infty$. If $s \ge 1$ for any $f \in H^{s-1}$ we have
$$
\|\blb \beta_\ld, |\nabla|^s \brb f\|_{L^2} \le C_\beta\ld^{-1}\|f\|_{H^{s-1}}.
$$
If $0 < s < 1$, then for any $f \in L^2$ we have
$$
\|\blb \beta_\ld, |\nabla|^s \brb f\|_{L^2} \le C_\beta\ld^{-s}\|f\|_{L^2}.
$$
\end{lemma}
From the above lemma it follows that
$$
\frac{d}{dt}\mlf^2 \le 2\ld\left(\sum_j\|\blb \frac{x_j}\ld\pl, |\nabla|^\al \brb u\|_{L^2}^2\right)^\frac12\mlf \le C\|u\|_{H^{\al-1}}\mlf \le CM\mlf
$$
and thus $\frac{d}{dt}\mlf \le CM$. Integrating over $[0, t]$, we have
$$
\mlf(t) \le CMt + \||x|\pl \varphi\|_{L^2}.
$$
Letting $\ld \to +\infty$, by Fatou's lemma we get the desired result.

Next we estimate $\mlft$ as follows.
\begin{align*}
\frac{d}{dt}\mlft^2 &= 2\im \bl |\nabla|^{-\frac\al2} x\pl |\nabla|(|\nabla|^\al - V(u)u) ; |\nabla|^\frac\al2x \pl |\nabla|u\br\\
&= 2\im \bl |\nabla|^{-\frac\al2}x\pl |\nabla|^{\al} |\nabla|u ; |\nabla|^\frac\al2 x\pl |\nabla| u\br - 2\im \bl x\pl |\nabla|(V(u)u) ; x\pl |\nabla|u\br\\
&= 2\sum_j\im \bl \blb x_j \pl, |\nabla|^\al \brb |\nabla|u, x_j\pl |\nabla|u\br\\
& \quad + 2\sum_j\im \bl x_j\pl \nabla|\nabla|^{-1}\cdot((\nabla V(u))u + V(u)\nabla u), x_j\pl |\nabla|u\br\\
&\le C\|u\|_{H^\al}\mlft + C\||x|((\nabla V(u))u + V(u)\nabla u)\|_{L^2}\mlft.
\end{align*}
For the last term we used the weight estimate of the singular integral operator $\nabla|\nabla|^{-1}$ with $A_2$-weight $|x|$.

If $V(u) = |u|^\frac{2\al}{d-\al}$, then by Sobolev inequality \eqref{sobo-radial} we have
\begin{align*}
\||x|((\nabla V(u))u + V(u)\nabla u)\|_{L^2} &\le C\||u|^\frac{2\al}{d-\al}|\nabla u|\|_{L^2} + \||x|^\al |u|^\frac{2\al}{d-\al}|\nabla u|\|_{L^2}\\
&\le C\||\nabla|^\frac\al2 u\|_{L^2}^\frac{2\al}{d-\al}\|u\|_{H^{1+\frac\al2}} + \||\nabla|^\frac\al2 u\|_{L^2}^\frac{2\al}{d-\al}\||\nabla|u\|_{L^2}
\end{align*}
and thus by integrating over $[0, t]$
$$
\mlft(t) \le \||x||\nabla|\varphi\|_{L^2} + C(1+M)^\frac{2\al}{d-\al}\int_0^t \|u(t')\|_{H^{1+\frac\al2}}\,dt'.
$$

If $V(u) = |x|^{-2\al}*|u|^2$, then from the fractional integration for radial function that \begin{align}\label{frac-int} |x|^\delta(|x|^{-\gamma}*|f|) \le C\||x|^{\delta-\gamma}f\|_{L^1}\;\; (0 <\delta \le \gamma < d-1)\end{align} we get
\begin{align*}
\||x||\nabla V(u)|u\|_{L^2} &\le \||\nabla V(u)|u\|_{L^2} + \||x|^{2\al}|\nabla V(u)|u\|_{L^2}\\
& \le \|\nabla V(u)\|_{L^\frac{2d}{\al}}\|u\|_{L^\frac{2d}{d-\al}} + C\|u\|_{L^2}^2\|\nabla u\|_{L^2}\\
&\le C\|u|\nabla u|\|_{L^\frac{2d}{2d-3\al}}\||\nabla|^\frac\al2 u\|_{L^2} + C\|\varphi\|_{L^2}^2\|\nabla u\|_{L^2}\\
&\le C\|u\|_{H^{1+\frac\al2}}^3
\end{align*}
and
\begin{align*}
\||x| V(u)|\nabla u|\|_{L^2} \le \|V(u)\|_{L^\infty}\|\nabla u\|_{L^2} + \|u\|_{L^2}^2\|\nabla u\|_{L^2} \le C\|u\|_{H^\al}^3.
\end{align*}
Thus $$\mlft(t) \le \||x||\nabla|\varphi\|_{L^2} + C\int_0^t \|u(t')\|_{H^{1+\frac\al2}}^3\,dt'.$$
Fatou's lemma yields the desired results.

Similarly to the estimate of $\mlf$ we have for $\mls$ that
\begin{align*}
\frac{d}{dt}\mls &= 2\im \bl |x|^2\pl (|\nabla|^\al u - V(u)u), |x|^2\pl u\br\\
&= 2\im \bl |x|^2\pl |\nabla|^\al u, |x|^2 \pl u\br = 2\im \bl \pl x \cdot (|\nabla|^\al x + \al |\nabla|^{\al-2}\nabla) u, |x|^2 \pl u\br\\
&= 2\im \bl x\pl  \cdot |\nabla |^\al x u,  |x|^2 \pl u\br + 2\al \im \bl \pl x \cdot \nabla/|\nabla||\nabla|^{\al-1} u, |x|^2 \pl u\br\\
&= 2\sum_{j}\im \bl \blb x_j\pl, |\nabla|^\al\brb x_ju, |x|^2 \pl u \br + 2\al \im \bl \pl x \cdot \nabla/|\nabla||\nabla|^{\al-1} u, |x|^2 \pl u\br.
\end{align*}
Lemma \ref{com-lem} shows that
$$
\frac{d}{dt}\mls \le C(\||x|u\|_{H^{\al-1}} + \||x||\nabla|^{\al-1}u\|_{L^2})\mls \le C(\|u\|_{H^{\al-1}} +  \||x||\nabla|^{\al-1}u\|_{L^2})\mls,
$$
which implies that $$\||x|^2u(t)\|_{L^2} \le \||x|^2\varphi\|_{L^2} + C\int_0^t (\|u\|_{H^{\al-1}} +  \||x||\nabla|^{\al-1}u\|_{L^2})\,dt'.$$
This completes the proof of Proposition \ref{moment}.

\end{proof}
\begin{proof}[Proof of Lemma \ref{com-lem}]
We show the first inequality. By Plancherel's theorem it suffices to show that $\|Tg\|_{L^2} \le C\|(1+ |\zeta|^{s-1})g\|_{L^2}$, where
$$
Tg(\xi) = \ld^d\int \widehat\beta(\ld(\xi-\zeta))(|\zeta|^s - |\xi|^s)g(\zeta)\,d\zeta.
$$
In fact,
\begin{align*}
|Tg(\xi)| &\le s\ld^d \int |\widehat \beta(\ld (\xi-\zeta))|(|\xi|^{s-1} + |\zeta|^{s-1})|\xi-\zeta||g(\zeta)|\,d\zeta\\
&\le s\ld^d \int |\widehat \beta(\ld (\xi-\zeta))||\xi-\zeta|^{s}|g(\zeta)|\,d\zeta + 2s\ld^d \int |\widehat \beta(\ld (\xi-\zeta))||\zeta|^{s-1}|\xi-\zeta||g(\zeta)|\,d\zeta\\
&= \ld^{-s}s\ld^d \int |\widehat \beta(\ld (\xi-\zeta))||\ld(\xi-\zeta)|^{s}|g(\zeta)|\,d\zeta\\
 &\quad + 2s\ld^d \int |\widehat \beta(\ld (\xi-\zeta))|\ld(\xi-\zeta)|||\zeta|^{s-1}g(\zeta)|\,d\zeta.
\end{align*}
Since $\widehat \beta(\xi)(|\xi|^\al + |\xi|)$ is integrable and $s > 1$, we get
$$
\|Tg\|_{L^2} \le C_\beta\ld^{-1}\|(1+|\zeta|^{s-1})g\|_{L^2}.
$$
Similarly for $0 < s < 1$ we have
$$
|Tg(\xi)| \le \ld^d \ld^{-s} \int |\widehat \beta(\ld (\xi-\zeta))|\ld(\xi-\zeta)|^{s-1}|g(\zeta)|\,d\zeta
$$
and thus
$$
\|Tg\|_{L^2} \le C_\beta\ld^{-s}\|g\|_{L^2}.
$$
This completes the proof of lemma.
\end{proof}

\subsection{Virial argument}
Here we consider the virial inequality through the moment estimates above.
Let us define two quantities associated with dilation and virial operators respectively by
$$
\mathcal A(u) := -{\rm Im}\big< u, x\cdot \nabla u \big> , \qquad \mathcal M(v) := \big<|\nabla|^{1-\frac\al2}(x u)\,;\, |\nabla|^{1-\frac\al2}(x u)\big>.
$$
From the regularity and moment estimates we can differentiate them w.r.t time.
\begin{align}\label{commut0}
\frac{d}{dt}\ald(u(t)) = {\rm Re}\big\langle (|\nabla|^\al u - V(u)u),\, x\cdot \nabla  u \big\rangle - {\rm Re}\bl  u,\, x\cdot \nabla (|\nabla|^\al u - V(u)u) \br.
\end{align}
By integration by parts, we have
\begin{align*}
\frac{d}{dt}\ald(u(t)) &= {\rm Re}\bl x |\nabla|^\al  u\,;  \nabla u \br  + d \bl u, V(u) u \br - {\rm Re}\bl u, x\cdot \nabla ( |\nabla|^\al u) \br\\
&\qquad + 2\re \bl u, x \cdot \nabla (V(u) u) \br.
\end{align*}

Using the identity $x |\nabla|^\beta = |\nabla|^\beta x + \beta |\nabla|^{\beta-2}\nabla $ for $0 < \beta < 2$, we have
\begin{align}\begin{aligned}\label{diff}
\frac{d}{dt}\ald(u(t)) &= \al \bl  u, |\nabla|^\al  u \br + 2 \re \bl  u, x\cdot \nabla (V(u) u) \br + d \bl  u, V(u) u \br.
\end{aligned}\end{align}

We first consider the power type. If $V(u) = |u|^\frac{2\al}{d-\al}$, then by direct calculation we get that
\begin{align*}
&2 \re \bl u, x\cdot \nabla (V(u) u) \br = - (d+\al) \re \bl  u , V(u)  u \br.
\end{align*}
Plugging this into \eqref{diff}, we have
\begin{align}\label{commut1}
\frac{d}{dt}\ald(u(t)) = \al(\blb \bl  u, |\nabla|^\al  u \br -  \bl  u, V(u) u \br\brb).
\end{align}
%where
%\begin{align*}
%\mathcal E_p(\ld) = -2 \re\bl \nabla \blb \pl, |\nabla|^\al\brb u ; x \pl u \br - n \re \bl \pl u, \blb \pl, |\nabla|^\al \brb u\br - \frac{n-\al}{n}\re \bl \pl u, V(u)(x \cdot \nabla \pl)u %\br.
%\end{align*}

Now we consider the Hartree case $V(u) = |x|^{-2\al}*|u|^2$. Using integration by parts, we also get
$$
2\re\bl  u, x\cdot \nabla(V(u)  u)\br = \re \bl  u, (x\cdot \nabla V(u)) u\br - d \bl  u, V(u) u \br.
$$
Since $d > 2\al + 1$, by direct differentiation we have
\begin{align*}
\re \bl  u, (x\cdot \nabla V(u)) u\br = - 2\al\bl  u, V(u)  u\br - 2\al\int\!\!\int  |u(x)|^2 |x-y|^{-2\al-1}y\cdot \frac{x-y}{|x-y|}|u(y)|^2\,dxdy.
 \end{align*}
In fact, from change of variables we deduce that
\begin{align*}
2\al\int\!\!\int |u(x)|^2 |x-y|^{-2\al-1}y\cdot \frac{x-y}{|x-y|}|u(y)|^2\,dxdy = \re\bl  u, (x\cdot \nabla V(u))  u\br .
\end{align*}
So, we have
\begin{align*}
\re \bl  u, (x\cdot \nabla V(u)) u\br = - \al\bl  u, V(u)  u\br.
 \end{align*}

Putting all together, we finally have
\begin{align}\label{commut2}
\frac{d}{dt}\ald(u(t)) = \al ( \blb\bl  u, |\nabla|^\al  u \br -  \bl  u, V(u) u \br \brb).
\end{align}
%where
%\begin{align*}
%\mathcal E_h(\ld) &= -2 \re\bl \nabla \blb \pl, |\nabla|^\al\brb u ; x \pl u \br - n \re \bl \pl u, \blb \pl, |\nabla|^\al \brb u\br\\
%&\qquad-\frac12\re\bl \pl u, (x\cdot \nabla V(\sqrt{1-\pl^2})u), \pl u \br - \frac12\re\bl\sqrt{1-\pl^2}u, (x\cdot \nabla V(\pl u)) \sqrt{1-\pl^2}u\br.
%\end{align*}

To deal with the RHS of \eqref{commut1} and \eqref{commut2} we introduce the following lemma to be shown in appendix.
\begin{lemma}\label{moment}
If $E(\varphi) \le (1-\delta_0) E(W_\al)$ and $\||\nabla|^\frac\al2 \varphi\|_{L^2} \ge \||\nabla|^\frac\al2 W_\al\|_{L^2}$ for some $0 < \delta_0 < 1$, then there exists a positive $\overline \delta$ such that $\||\nabla|^\frac\al2 u(t)\|_{L^2}^2 \ge (1 + \overline \delta)\||\nabla|^\frac\al2 W_\al\|_{L^2}^2$ for all $t \in (0, T^*)$.
\end{lemma}

From Proposition \ref{trap} it follows that
\begin{align*}
\bl  u, |\nabla|^\al  u \br -  \bl  u, V(u) u \br &= \mu\left(\frac12\||\nabla|^\frac\al2 u\|_{L^2}^2 - \frac1{\mu}\int V(u)|u|^2\,dx \right) - \frac{\mu-2}2\||\nabla|^\frac\al2 u\|_{L^2}^2\\
& = \mu E(\varphi) - \frac{\mu-2}2\||\nabla|^\frac\al2 u\|_{L^2}^2) \le \mu E(W_\al) - \frac{\mu-2}2(1 + \overline \delta)\||\nabla|^\frac\al2 W_\al\|_{L^2}^2\\
 &= - \frac{\mu-2}2 \,\overline\delta \,C_{d, \al}^{-\frac2{\mu-2}} =: -\epsilon_0 < 0.
\end{align*}
Thus integrating \eqref{commut1} and \eqref{commut2} over $[0,t]$ we get
\begin{align}\label{dilation}
\mathcal A(u(t)) \le  \mathcal A(\varphi) - \al\epsilon_0 t.
\end{align}

On the other hand by differentiating $\mathcal M$ and using the identity $x |\nabla|^\beta = |\nabla|^\beta x + \beta |\nabla|^{\beta-2}\nabla $ for $\beta = \al$ and $2-\al$, we get
\begin{align*}
\frac{d}{dt}\mathcal M(u) &= 2 \im \bl |\nabla |^{1-\al}x(|\nabla |^\al u - V(u)u) ; |\nabla |x u\br\\
&= 2\im \bl |\nabla |^{1-\al} x |\nabla |^\al u ; |\nabla |x u \br - 2\im \bl x V(u)u ; |\nabla|^{2-\al}x u\br\\
&=  - 2\al\im \bl u, x\cdot \nabla u \br - 2\im \bl |x|^2 V(u)u ; |\nabla |^{2-\al}u\br - 2(2-\al)\im \bl x V(u)u ; |\nabla |^{-\al}\nabla u \br\\
&= 2\al \mathcal A(u) - 2\im \bl |x|^2 V(u)u , |\nabla |^{2-\al}u\br - 2(2-\al)\im \bl x V(u)u ; |\nabla |^{-\al}\nabla u \br.
\end{align*}
Since $|x|^2u \in L^2$, $V(u) \in L^\infty$ and $u \in H^{1 + \frac\al2}$, the second term of last line is at least well-defined.
Actually, it is possible to get a better estimate as below.

If $V(u) = |u|^\frac{2\al}{d-\al}$, then since $\frac 43 \le \al < 2$ we have
\begin{align*}
- 2\im \bl |x|^2 V(u)u ; |\nabla |^{2-\al}u\br &\le 2\||x|^\al V(u)\|_{L^\infty}\||x|^{2-\al}u\|_{L^2}\||\nabla|^{2-\al}u\|_{L^2}\\
&\le Cm^{\frac{\al^2 + 2\al -4}{\al}}M^{\frac{2\al}{d-\al}+\frac{4-2\al}{\al}}\||x|u\|_{L^2}^{2-\al}.
\end{align*}
From Lemma \ref{com-lem} it follows that
$$
- 2\im \bl |x|^2 V(u)u ; |\nabla |^{2-\al}u\br \le C m^{\frac{\al^2 + 2\al -4}{\al}}M^{\frac{2\al}{d-\al}+\frac{4-2\al}{\al}}(Mt + m_1)^{2-\al},
$$
where $m_1 = \||x|\varphi\|_{L^2}$.

On the other hand, the last term is bounded by
\begin{align}\label{sec}
C\int |\nabla|^{-(\al-1)}(|\cdot|^{-(\al-1)}g)(x)f(x)\,dx = C\int\!\!\int \frac{f(x)g(y)}{|x-y|^{-(d-(\al-1))}|y|^{\al-1}}\,dxdy,
\end{align}
where $f = |\nabla/|\nabla| u|$ and $g = x |x|^{\al-1} V(u)u$. For this we use Stein-Weiss inequality that
\begin{align}\label{s-t}
\left|\int\!\int \frac{f(x)\overline g(y)}{|x|^{\theta_1}|x-y|^{\theta}|y|^{\theta_2}}\,dxdy\right| \le C\|f\|_{L^{p_1}}\|g\|_{L^{p_2}},
\end{align}
provided that $1 < p_1, p_2 < \infty$, $\theta_1 + \theta_2 \ge 0$, $0 < \theta < d$, $\frac1{p_1} + \frac1{p_2} + \frac{\theta+\theta_1+\theta_2}{d} = 2$ and $\theta_1 < \frac d{p_1'}$, $\theta_2 < \frac{d}{p_2'}$.
Let $p_1 = p_2 = 2$ and $\theta_1 = 0, \theta_2 = \al-1, \theta = d-(\al-1)$. Then \eqref{s-t} implies that
$$
\eqref{sec} \le Cm^2M^\frac{2\al}{d-\al}.
$$
These estimates lead us to
\begin{align}\label{virial-p}
\mathcal M(u(t)) \le  -\al^2 \epsilon_0t^2 + \left(C(m, M)(Mt + m_1)^{3-\al} + (Cm^2M^\frac{2\al}{d-\al} + \mathcal A(\varphi))t\right) + \mathcal M(\varphi).
\end{align}

We then consider the Hartree case. We follow the same strategy as in \cite{chkl}. To begin with
let us observe that
$$
2\im \bl |x|^2 V(u)u , |\nabla |^{2-\al}u\br = \im \bl \blb |\nabla|^{2-\al}, g\brb u, u \br,
$$
where $g = |x|^2V(u)$. Then by the commutator estimate of \cite{chkl} one can get
$$
\|\blb |\nabla|^{2-\al}, g\brb u\|_{L^2} \le C m \sup_{x\neq y}\frac{|g(x) - g(y)|}{|x-y|^{2-\al}}.
$$
If $x\neq y$, then
$$
|g(x) - g(y)| \le |x-y|\int_0^1 |\nabla g(z_s)|\,ds, \;\;z_s = x + s(y-x).
$$
Since $|\nabla g(z_s)| \le 2|z_s|V(u) + 2\al |z_s|^2\int |z_s-y|^{-(2\al+1)}|u(y)|^2\,dy$ and $d > 2\al+2$, by \eqref{frac-int} and Hardy-Sobolev inequality we have
$$
|\nabla g(z_s)| \le C|z_s|^{1-\al}(\||x|^{-2\al+\al}|u|^2\|_{L^1} + \||x|^{-2\al-1+1+\al}|u|^2\|_{L^1}) \le C|z_s|^{1-\al}\||\nabla|^\frac\al2 u\|_{L^2}^2 \le CM^2|z_s|^{1-\al}.
$$
Thus $|g(x) - g(y)| \le C|x-y|^{2-\al}M^2$, which implies that
$$
|\im \bl |x|^2 V(u)u , |\nabla |^{2-\al}u\br| \le CM^2m^2.
$$
Moreover, since by \eqref{frac-int} $|x|^\al V(u) \le C\||x|^{-\al}|u|^2\|_{L^2} \le CM^2$, from \eqref{sec} and \eqref{s-t} we have
$$
|\im \bl x V(u)u ; |\nabla |^{-\al}\nabla u \br| \le CM^2m^2.
$$
Therefore we get
\begin{align}\label{virial-h}
\mathcal M(u(t)) \le -\al^2\epsilon_0 t^2 + (Cm^2M^2 + \mathcal A(\varphi)) t + \mathcal M(\varphi).
\end{align}
Since $\mathcal M(u)$ is non-negative, by \eqref{virial-p} and \eqref{virial-h} we deduce that $T^* < +\infty$.

\section{Appendix}
We consider the characterization of maximizer of \eqref{bc} only for Hartree equation. For this we study a minimization problem:
\begin{align}\label{min1}
m = \inf_{u\in H^\frac\al2,\;\; \int V(u)|u|^2dx \neq 0} I(u), \quad I(u) := \frac{\||\nabla|^\frac\al2 u\|_{L^2}^4}{\int V(u)|u|^2dx}.
\end{align}
This is equivalent to the constrained minimization problem:
\begin{align}\label{min2}
m = \inf_{u\in H^\frac\al2,\;\; \int V(u)|u|^2dx = 1} J(u),\quad J(u):=\||\nabla|^\frac\al2 u\|_{L^2}^4.
\end{align}
By Sobolev embedding one can observe that $m > 0$. Suppose that $\underline u \in H^\frac\al2$ is a minimizer of \eqref{min2}. Then since $J$ is Fr\'{e}chet differentiable on $H^\frac\al2$, for any $\phi \in C_0^\infty$
$J$ should satisfy that
$$
\frac{d}{d\ep}J(v_\ep)\Big|_{\ep=0} = 0, \;\;\mbox{where}\;\; v_\ep = \frac {\underline u + \ep \phi}{\left(\int V(\underline u + \ep \phi)|\underline u + \ep \phi|^2\,dx\right)^\frac14}.
$$
By direct calculation we conclude that
$$
\bl |\nabla|^\al \underline u - m^\frac12 V(\underline u)\underline u,  \phi\br = 0.
$$
which means $\underline u$ is a solution to $|\nabla|^\al w - m^\frac12V(w)w$. By using a change of variables it is also a solution to \eqref{ell eqn}. Thus the minimizer $\underline u$ is $e^{i\theta}\lambda^\frac{d-\al}{2}W_\al(\lambda(x-x_0))$. Here we note that $W_\al \in H^\frac\al2$ because $d > 2\al$. Now it remains to show that $J$ attains $m$ in $H^\frac\al2$.
In fact, the minimizer can be found in $H^\frac\al2_{rad}$. Choose a minimizing sequence $u_j \in H^\frac\al2_{rad}$ with $\int V(u_j)|u_j|^2\,dx = 1$. Then it is bounded in $H^\frac\al2_{rad}$ and thus we can take a subsequence converging weakly to $\underline u$. According to Lemma 5.2 of \cite{dass} $\int V(u_j)|u_j|^2\,dx \to \int V(u)|u|^2\,dx$ due to the radial symmetry, which implies that $\int V(\underline u)|\underline u|^2\,dx = 1$. By the lower semi-continuity we deduce that $m \le \||\nabla|^\frac\al2 \underline u\|_{L^2}^4 \le \liminf_{j \to \infty}\||\nabla|^\frac\al2 u_j\|_{L^2}^4 = m$. Therefore $\underline u$ is a minimizer.

\end{document}